\documentclass[reqno]{amsart}
\usepackage{latexsym,amsfonts,amssymb,amsmath,amsthm,amscd,enumitem}

\usepackage[usenames,dvipsnames]{xcolor}
\usepackage[colorlinks=true]{hyperref}
\usepackage{tikz-cd}
\usepackage{soul}  
\usepackage{stackengine} 
\usepackage[normalem]{ulem}

\usepackage{mathtools}

\usepackage{enumitem}

\usepackage{mdwlist}

\usepackage{comment}
\usepackage[dvipsnames]{xcolor}

\usepackage{bbm}
\usepackage{hyperref}

\def\dotfill#1{\cleaders\hbox to #1{.}\hfill}
\makeatletter
\def\myrulefill{\leavevmode\leaders\hrule height .7ex width 1ex depth -0.6ex\hfill\kern\z@}
\makeatother 

\theoremstyle{plain}
\newtheorem{theorem}{Theorem}[section]
\newtheorem{remark}{Remark}[section]
\newtheorem{lemma}{Lemma}[section]
\newtheorem{proposition}{Proposition}[section]

\newtheorem{definition}{Definition}[section]
\newtheorem{example}{Example}[section]

\numberwithin{equation}{section}
\numberwithin{definition}{section}

\DeclareFontFamily{U}{mathx}{\hyphenchar\font45}
\DeclareFontShape{U}{mathx}{m}{n}{
      <5> <6> <7> <8> <9> <10>
      <10.95> <12> <14.4> <17.28> <20.74> <24.88>
      mathx10
      }{}
\DeclareSymbolFont{mathx}{U}{mathx}{m}{n}
\DeclareFontSubstitution{U}{mathx}{m}{n}
\DeclareMathAccent{\widecheck}{0}{mathx}{"71}
\DeclareMathAccent{\wideparen}{0}{mathx}{"75}

\newcommand\Periwinklesout{\bgroup\markoverwith{\textcolor{Periwinkle}{\rule[0.5ex]{2pt}{0.4pt}}}\ULon}
\newcommand\violetsout{\bgroup\markoverwith{\textcolor{violet}{\rule[0.5ex]{2pt}{0.4pt}}}\ULon}

\newcommand\RR{\ensuremath{\mathbb{R}}}
\newcommand\QQ{\ensuremath{\mathbb{Q}}}

\newcommand\NN{\ensuremath{\mathbb{N}}}
\newcommand\PP{\ensuremath{\mathbb{P}}}

\newcommand{\OFP}{\ensuremath{(\Omega,\mathfrak{F},\PP)}}

\newcommand{\norm}[1]{\ensuremath{\lVert#1\rVert}}

\newcommand{\dd}{\ensuremath{\mathrm{d}}}

\newcommand{\vertiii}[1]{{\left\vert\kern-0.25ex\left\vert\kern-0.25ex\left\vert #1 
    \right\vert\kern-0.25ex\right\vert\kern-0.25ex\right\vert}}

\newcommand{\WST}{weak\nobreakdash-\hspace{0pt}\textup{*} }

\definecolor{afb}{rgb}{0.36, 0.54, 0.66}
\definecolor{o}{RGB}{242, 138, 2}
\definecolor{bur}{rgb}{0.5, 0.0, 0.13}
\definecolor{ph}{rgb}{0.87, 0.0, 1.0}
\definecolor{pers}{rgb}{0.0, 0.65, 0.58}

\begin{document}

\title[Delay equations with measure-valued coefficients]{ Delay equations with measure-valued coefficients:
From Carath\'eodory solutions to measurable semiflows }
\author[M.~Kryspin]{Marek Kryspin}
\address[Marek Kryspin]{Faculty of Pure and Applied Mathematics,
Wroc{\l}aw University of Science and Technology,
Wybrze\.ze Wyspia\'nskiego 27, PL-50-370 Wroc{\l}aw, Poland.}
\email[Marek Kryspin]{marek.kryspin@pwr.edu.pl}
\begin{abstract}
 We study linear delay differential equations whose delay
terms are represented by time-dependent finite signed
measures. This framework accommodates discrete,
distributed, and singular delay contributions without
requiring continuity of the coefficient paths in time.
We establish well-posedness in the Carath\'eodory sense
and sequential dependence of solutions on the
measure-valued parameters under almost-everywhere
weak-* convergence and a common total variation bound.

A central part of the analysis associates a signed
measure on the time--delay domain with each coefficient
path and establishes an equivalent integral formulation.
This representation yields three equivalent
characterizations of a $\sigma$-algebra on the parameter
space suited to measurability of the continuation
operator. A measurable fixed-point argument then
provides measurability of the solution maps, despite
the possible failure of continuity in the product
weak-* topology.

For coefficients generated by a measure-preserving
base flow, these results lead to measurable linear
skew-product semiflows on continuous and absolutely
continuous phase spaces. Under additional ergodic
hypotheses and the assumption of a finite top
Lyapunov exponent, eventual compactness and
regularization yield Oseledets decompositions with
the same Lyapunov exponents and multiplicities
on both spaces.
\end{abstract}
\maketitle

\section*{Introduction} 

Linear delay differential equations describe evolutionary processes
whose present rate of change depends on a history of previous
states. A measure on the history interval provides a convenient
representation of this dependence, allowing discrete and
distributed delays to be treated within the same framework.
In this paper, we consider equations of the form
\begin{equation} \label{main-intro-eq}
    z'(t)
    =
    \int_{[-1,0]} z(t+s)\,\dd\mu_t(s),
\end{equation}
on a finite time interval $[0,T]$, with $T>0$, subject to
the initial condition
\begin{equation*}
    z(s)=z_0(s),
    \qquad s\in[-1,0],
\end{equation*}
where $\mu_t$ is a finite signed Borel measure on $[-1,0]$.
Solutions are understood in the Carath\'eodory sense:
they are absolutely continuous on $[0,T]$ and satisfy
the equation almost everywhere.

The measure representation includes several familiar models.
The equation with a fixed delay
\begin{equation*}
z'(t)=a(t)z(t)+b(t)z(t-1)
\end{equation*}
is obtained by choosing
\begin{equation*}
\mu_t=a(t)\delta_0+b(t)\delta_{-1}.
\end{equation*}
Replacing the fixed atom by a moving one,
$\mu_t=a(t)\delta_0+b(t)\delta_{-r(t)}$, yields
\begin{equation*}
z'(t)=a(t)z(t)+b(t)z(t-r(t)),
\end{equation*}
where $r(t)\in[0,1]$. In particular, the choice
$r(t)=t-\lfloor t\rfloor$ gives equations with piecewise constant
argument, studied classically by Cooke and Wiener
\cite{CookeWiener1984} and more recently by Jaur\'e and Maul\'en
\cite{JaureMaulen2026} in the context of remotely almost periodic
solutions. Absolutely continuous components of the measures describe
distributed memory. The theory of Volterra integro-differential
equations developed by Gripenberg, Londen and Staffans
\cite{GripenbergLondenStaffans1990} provides a related setting.
Its finite-memory subclass includes
\begin{equation*}
z'(t)=a(t)z(t)+\int_0^t k(t-r)z(r)\,\dd r,
\end{equation*}
where $k\in L_1([0,1])$ is extended by zero outside $[0,1]$.
This equation is represented by
\begin{equation*}
\mu_t
=
a(t)\delta_0+
\mathbbm{1}_{[-\min\{t,1\},0]}(s)k(-s)\,\dd s.
\end{equation*}
Another connection is provided by the renewal equation
\begin{equation*}
z(t)=\int_0^1 k(r)z(t-r)\,\dd r,
\end{equation*}
whose dynamical formulation is discussed by Diekmann and
Verduyn Lunel \cite[Section~7]{DiekmannVerduynLunel2021}.
For $k\in W^{1,1}([0,1])$, differentiation yields
\eqref{main-intro-eq} with
\begin{equation*}
\mu=k(0)\delta_0-k(1)\delta_{-1}+k'(-s)\,\dd s.
\end{equation*}
Finally, certain equations with proportional delay can be transformed
into equations with fixed delay by a logarithmic change of time.
Kulenovi\'c \cite[Remark~2]{Kulenovic1995} describes this reduction
for first-order Euler equations with delay.

These examples fit naturally into the classical theory of
functional differential equations. As developed in particular
by Hale and Verduyn Lunel~\cite{Hale-Le}, this theory provides
the analytical foundations for initial-value problems and their
solution operators. Earlier contributions include the
well-posedness theory for hereditary systems developed by
Hale and Cruz~\cite{HaleCruz1970}, and the work of Miller and
Sell~\cite{MillerSell1968,MillerSell1970} on existence,
uniqueness, and continuity of solutions of integral equations.
The representation of bounded linear delay functionals by
measures is already part of this classical framework.

The structural and duality aspects of time-varying retarded
functional differential equations were studied by Colonius,
Manitius, and Salamon~\cite{ColoniusManitiusSalamon1989}.
Their work provides a further point of reference for the
operator-theoretic treatment of non-autonomous delay equations.

In the present work, we regard the entire measure-valued
coefficient path as a parameter and study the dependence
of the solution on this path. We then consider coefficient
paths generated by a measure-preserving base flow.
This connects deterministic well-posedness and parameter
dependence with the measurability needed to construct
a linear skew-product semiflow.

Continuous dependence on parameters has a long history
in the theory of ordinary differential equations, with
contributions by Petrov~\cite{Petrov1962,Petrov1964},
Opial~\cite{Opial1967}, and Heunis~\cite{Heunis1984}.
These works provide historical context for the
parameter-dependence questions considered here.

More directly related to measure-valued parameters is the
work of Banks, Dediu, and Nguyen~\cite{BanksDediuNguyen2007},
who develop a sensitivity theory for dynamical systems
with parameters in convex subsets of topological vector spaces.
Their applications include distributed delay equations
depending on probability measures and approximations of
such measures in the Prohorov topology.
Our analysis concerns time-dependent signed measures,
sequential continuity of the associated solution maps,
and a measurable structure on the space of coefficient paths.
The priority is on convergence and measurability of solutions,
rather than differentiability with respect to the measure
parameter.

A further point of comparison is the work of Longo, Novo,
and Obaya~\cite{LongoNovoObaya2019} on non-autonomous
Carath\'eodory delay differential equations.
They introduce and study topologies of integral type
that yield continuous dependence and continuous
skew-product semiflows, including results on Sobolev
phase spaces. Their results demonstrate how the topology
of the coefficient space can be adapted to the evolution.
Here, the parameter space is expressed directly in terms
of measure-valued paths. We examine the distinction between
sequential continuity in the relative product weak-*
topology and measurability with respect to a
$\sigma$-algebra generated by integrated delay functionals.
The resulting construction provides a measure-based
counterpart to the study of coefficient spaces through
integral topologies.

Related questions also arise in the theory of generalized
ordinary differential equations and Kurzweil--Stieltjes
integration. The connection between generalized differential
equations and continuous dependence on parameters goes
back to Kurzweil~\cite{Kurzweil1957}.
The contemporaneous work of Kurzweil and
Vorel~\cite{KurzweilVorel1957} also forms part of
the early development of parameter-dependence theory
for ordinary differential equations.
Federson, Mesquita, and
Slav\'ik~\cite{FedersonMesquitaSlavik2012} establish
existence, uniqueness, and continuous-dependence results
for measure functional differential equations and relate
them to functional dynamic equations on time scales.
Slav\'ik~\cite{Slavik2015} develops further well-posedness
results for abstract generalized differential equations
and their functional counterparts.
A typical equation in this setting is
\begin{equation*}
    y(t)
    =
    y(t_0)
    +
    \int_{t_0}^{t} f(y_\tau,\tau)\,\dd g(\tau),
\end{equation*}
where the time integral is understood in the
Kurzweil--Stieltjes sense.
In our formulation, the measure acts on the delay variable,
and the outer time integral is a Lebesgue integral.
The two settings therefore place the measure dependence
at different levels of the equation, while sharing an
interest in integral formulations under limited regularity.

The assumptions used for the deterministic equation
are stated directly in terms of weak-* measurability
of the coefficient path and bounds on its total variation.
They allow delay measures without densities and coefficient
paths without time continuity. The precise conditions,
including those needed for convergence of sequences of
parameters, are given with the corresponding results.
This separation of hypotheses is useful because
well-posedness of an individual equation, convergence
under variation of its coefficients, and the construction
of a random dynamical system impose different requirements.

An important feature of the parameter space is that
sequential continuity need not imply continuity.
Under a common bound on total variation, pointwise weak-*
convergence of a sequence of coefficient paths gives
convergence of the associated integral operators.
However, these operators need not be continuous for the
relative product topology on the space of measurable paths.
Indeed, a basic product neighbourhood constrains the path
at only finitely many times, whereas integration depends
on its values over a time interval.
Consequently, sequential convergence results alone
do not settle the measurable properties required for
the random evolution.

To address these properties, we associate with each
admissible coefficient path on $[0,T]$ a finite signed
measure $\nu$ ($\nu=\nu(\mu)$) on $Q=[0,T]\times[-1,0]$.
This measure is characterized by
\begin{equation*}
    \int_{Q}\varphi(\tau,s)\,\dd\nu(\tau,s)
    =
    \int_0^T
    \int_{[-1,0]}
    \varphi(\tau,s)\,\dd\mu_\tau(s)\,\dd\tau,
    \qquad \varphi\in C(Q).
\end{equation*}
We establish the integration identity that allows the
Carath\'eodory equation to be written as
\begin{equation*}
    z(t)
    =
    z_0(0)
    +
    \int_{[0,t]\times[-1,0]}
    z(\tau+s)\,\dd\nu(\tau,s).
\end{equation*}
The significance of this representation for the present
work lies in its connection with the parameter space:
the measure on the product interval provides a common
object through which the delay functionals and their
measurability can be studied.

Let $\mathfrak m_T$ denote the space of admissible
measure-valued paths on $[0,T]$.
We identify a $\sigma$-algebra $\Sigma_T$ through
three equivalent descriptions: measurability of the
integrated delay functionals, weak-* measurability
of $\mu\mapsto\nu(\mu)$, and measurability of this map
with respect to the weak-* Borel structure.
The equivalence connects test functions of the form
$x(\tau+s)$, which occur naturally in the equation,
with arbitrary continuous test functions on $Q$.

This measurable structure is then used to prove
measurability of the continuation operator for each
fixed argument. Continuity in the function variable
allows measurable successive approximations to be
constructed, and their limit gives a measurable solution
map. Thus, the passage from coefficients to solutions
is justified within the same parameter framework
used to formulate the equation.

The interpretation of non-autonomous equations through
skew-product constructions has a substantial history.
Artstein~\cite{Artstein1977ODE} studies the dynamics
associated with translates of an ordinary differential
equation and gives conditions of integral type ensuring
compactness of the resulting construction.
In \cite{Artstein1977Kurzweil}, the analysis is extended
to limiting equations that may belong to the generalized
class introduced by Kurzweil.
These works illustrate the importance of choosing
a coefficient space compatible with translation
and passage to limits.
In the present paper, the corresponding issue is
the measurable structure of measure-valued coefficient
paths and its compatibility with the solution operator.

The dynamical part of the paper follows the cocycle
framework presented in Arnold's
monograph~\cite{Arnold1998}. In this framework,
a measure-preserving base flow describes the evolution
of the random environment, and a measurable cocycle
describes the corresponding evolution in the phase space.
For coefficients generated by $\mu_t=\mu_{\theta_t\omega}$, 
the solution operators satisfy the cocycle identity
\begin{equation*}
    U_\omega(t+s)
    =
    U_{\theta_s\omega}(t)\,U_\omega(s),
    \qquad t,s\geq0.
\end{equation*}
The algebraic identity follows from uniqueness.
The measurable cocycle structure additionally requires
control of the dependence on $\omega$, which is obtained
from the assumption that
\begin{equation*}
    \Omega\ni\omega
    \mapsto 
    (\mu_{\theta_t\omega})_{t\in[0,T]}
    \in\mathfrak m_T
\end{equation*}
is $(\mathfrak{F},\Sigma_T)$-measurable for every $T>0$.
This is the point at which the analysis of the coefficient
space enters the random dynamical system construction.

A classical spectral perspective on linear skew-product
dynamics is provided by Sacker and
Sell~\cite{SackerSell1978}, whose theory relates
spectral structure to exponential dichotomies.
This provides a complementary background to the
measure-theoretic description of growth rates and
invariant subspaces used in the Oseledets theory below.

The asymptotic theory of random linear evolutions provides
a further motivation for this construction.
Mierczy\'nski and Shen~\cite{MierczynskiShen2013}
develop an abstract theory of generalized principal
Lyapunov exponents, principal Floquet spaces, and
exponential separation for positive random dynamical
systems. Their applications include cooperative random
delay systems~\cite{MierczynskiShen2016}.
Mierczy\'nski, Novo, and
Obaya~\cite{MierczynskiNovoObaya2018}
develop exponential separation of type~II in a random
setting adapted to delay equations.
These results concern distinguished positive invariant
directions and their separation from complementary
directions, under focusing
assumptions.

The relation between this theory and Oseledets
decompositions is further developed by Kryspin,
Mierczy\'nski, Novo, and
Obaya~\cite{Marek-Janusz-Sylvia-Rafa},
who provide two approaches to exponential separation
for random delay systems, one of them using an
Oseledets decomposition. Their work also includes
a realization on an absolutely continuous phase space. Our contribution concerns the
measure-valued coefficient setting and the verification
of the analytical and measurable properties needed
to apply this dynamical framework.

For the Oseledets analysis, a particularly close
reference is Mierczy\'nski, Novo, and
Obaya~\cite{MierczynskiNovoObaya2020}.
They study random systems with a fixed discrete delay
and compare Lyapunov exponents and Oseledets
decompositions on continuous and integrable phase
spaces. Here, the delay term is described by an
admissible time-dependent signed measure, and the
phase spaces under consideration are
$C([-1,0])$ and $AC([-1,0])$.
The abstract results of
Kryspin~\cite{Kryspin2024} provide a framework
for comparing decompositions on an ambient space
and a more regular invariant subspace.

The use of signed delay measures does not itself
impose positivity of the evolution.
Consequently, the Oseledets decomposition considered
here must be distinguished from the existence of
a principal Floquet space or an exponential separation,
which requires additional structural assumptions.
Under the ergodic and integrability hypotheses stated
in the final section, we verify the conditions needed
for the Oseledets analysis and use the regularization
of solution segments to relate the two phase-space
realizations.

The contribution of the paper is therefore a connected
analysis of measure-valued delay coefficients:
from deterministic well-posedness and parameter
dependence, through an integrated-measure representation
and a characterization of the relevant measurable
structure, to a random linear evolution and its
Oseledets decomposition.
This connection is necessary because the asymptotic
theory applies to a measurable cocycle, whereas
the equation is initially specified only through
a family of delay measures.

The paper is organized as follows. Section~\ref{sec:Preliminaries} fixes the notation and recalls the
measure-theoretic conventions used throughout the paper.
Section~\ref{sec:Main-eq} introduces the equation and the standing
assumptions.
Section~\ref{sec:Carath-sol} establishes existence and uniqueness in the \allowbreak
Carath\'eodory setting and studies dependence on initial
data and coefficient paths.
Section~\ref{sec:Measure-integration-and-disintegration} constructs the integrated measure and proves
the equivalence of the two integral formulations. Section~\ref{sec:Integral-formulation} revisits the contraction argument using the integrated measure
and its pushforward, providing an alternative proof of existence and
uniqueness within this formulation.
Section~\ref{sec:skew-product} develops the skew-product construction,
with the measurability assertions deferred to
Section~\ref{sec:Measurable-structures}.
That section identifies the relevant Borel structures
and equivalent descriptions of $\Sigma_T$, and proves
measurability of the solution operators.
The final section~\ref{sec:Oseledets} concerns the Oseledets decomposition
and its relation to the choice of phase space.

\section{Preliminaries and notation} 
\label{sec:Preliminaries}
\smallskip\par
For a topological space $S$, we denote its Borel
$\sigma$-algebra by $\mathfrak{B}(S)$. Lebesgue measure is denoted by $\ell$, and $\mathfrak{L}(I)$ denotes the Lebesgue $\sigma$-algebra
on an interval $I$. Given a set $X$ and a family of mappings
$f_i:X\to Y_i$, $i\in J$, where each $(Y_i,\Sigma_i)$
is a measurable space, we write
\begin{equation*}
    \sigma(f_i:i\in J)
    :=
    \sigma(
        \{f_i^{-1}(A):i\in J,\ A\in\Sigma_i\}
    )
\end{equation*}
for the smallest $\sigma$-algebra on $X$ making
all the mappings $f_i$ measurable.
When the target spaces are topological spaces,
their Borel $\sigma$-algebras are understood
unless stated otherwise.

For a signed Borel measure $\alpha:\mathfrak{B}(S)\to\RR$, we define the total variation measure $|\alpha|$ as $\alpha^++\alpha^-$ where the pair $(\alpha^+,\alpha^-)$ is the Hahn\nobreakdash--\hspace{0pt}Jordan decomposition. For general case of measure space $(X,\Sigma,\alpha)$ with a signed measure $\alpha$, by the \emph{positive} set we understand the unique (up to null sets) $P^{\alpha}\in\Sigma$ satisfying one of the following (equivalent) definitions:  for any $A\in\Sigma$ we have $\alpha^{+}(A)=\alpha(A\cap P^{\alpha})$, $\alpha\ge 0$ on $P^{\alpha}$ and $\alpha\le 0$ on $X\setminus P^{\alpha}$ or simply  $\alpha(P^{\alpha})=\sup\{\alpha(A):A\in\Sigma\}$. In an analogous way we define the \emph{negative} set $N^{\alpha}$. Superscripts will denote the measure for which we consider whether a set is positive or negative. We wish to retain them, as later we will have many positive and negative sets simultaneously for different measures.
For a finite signed measure $\alpha$, we write
\begin{equation*}
    \|\alpha\|_{\mathrm{TV}}:=|\alpha|(S),
\end{equation*}
where $S$ is the underlying measurable space.
For $1\leq p\leq\infty$, the space $L_p(\alpha)$ is
understood in the standard sense as a space of equivalence
classes of measurable real-valued functions, identified
when they agree $|\alpha|$-almost everywhere.
It is a Banach space with the norm
\begin{equation*}
    \|f\|_{L_p(\alpha)}
    :=
    \Big(\int |f|^p\,\dd|\alpha|\Big)^{1/p},
    \qquad 1\leq p<\infty,
\end{equation*}
and
\begin{equation*}
    \|f\|_{L_\infty(\alpha)}
    :=
    |\alpha|\text{-}\operatorname*{ess\,sup}|f|.
\end{equation*}
All $\alpha$-almost-everywhere statements are likewise
understood with respect to $|\alpha|$.
In particular, an $\alpha$-null set is a measurable
set $A$ satisfying $|\alpha|(A)=0$.

For Banach spaces $X$ and $Y$, we denote by
$\mathcal L(X,Y)$ the space of bounded linear operators
from $X$ to $Y$, equipped with the operator norm.
We write $\mathcal L(X):=\mathcal L(X,X)$. 
By $X^*$ we denote the dual of a Banach space $X$, we denote by $\langle \cdot, \cdot \rangle_{X, X^*}$ the duality pairing. We omit subscripts, as this does not lead to misunderstandings. The topology on $X^*$ will be the standard norm topology unless stated otherwise. Most often, we will be dealing with the space of continuous functions here denoted by $C(I)$, where $I\subset \RR$. We shall also use the space $AC(I)$ of absolutely continuous
functions on $I$. Throughout, these spaces consist of
real-valued functions. For a compact interval $I=[a,b]$,
we equip $C(I)$ with the supremum norm
\begin{equation*}
    \|u\|_{C(I)}:=\sup_{t\in I}|u(t)|,
\end{equation*}
and $AC(I)$ with the norm
\begin{equation*}
    \|u\|_{AC(I)}
    :=
    |u(a)|+\int_a^b |u'(t)|\,\dd t.
\end{equation*}
Both spaces are Banach spaces with these norms, and
\begin{equation*}
    \|u\|_{C(I)}\leq\|u\|_{AC(I)},
    \qquad u\in AC(I).
\end{equation*}
Unless explicitly stated otherwise, all references to
norms, convergence, and continuity in these spaces
are understood with respect to the norms defined above.

For $T>0$ and a fixed initial history $u_0\in C$, define
\begin{equation*}
    C_{u_0}([0,T])
    :=
    \{v\in C([0,T]):v(0)=u_0(0)\}
\end{equation*}
and
\begin{equation*}
    AC_{u_0}([0,T])
    :=
    \{v\in AC([0,T]):v(0)=u_0(0)\}.
\end{equation*}
These are nonempty closed affine subsets of their
respective Banach spaces, equipped with the induced
metrics. In particular, both are complete.
Their definitions depend on $u_0$ only through
its value at zero; the full history enters through
the extension defined below.

\smallskip\par
For $z\in C([-1,T])$ and $t\in[0,T]$, its history
segment $z_t\in C$ is defined by
\begin{equation*}
    z_t(s):=z(t+s),
    \qquad s\in[-1,0].
\end{equation*}
The notation $z_t$ always refers to this restricted
segment, rather than to the entire translated function.

We write $z(\cdot;\mu,u_0)$ for the solution corresponding
to the coefficient path $\mu$ and initial history $u_0$.
When the coefficient path is generated by a base point
$\omega$, we write $z(\cdot;\mu,u_0)$ as well and understand it as a deterministic solution $z(\cdot;(\mu_{\theta_t\omega})_{t\in[0,T]},u_0)$.

The solution operators on $C$ and $AC$ are denoted by
$U_\omega^C(t)$ and $U_\omega^{AC}(t)$, respectively.
The superscript is omitted whenever the phase space
is clear from the context.

\section{Main equation and assumptions} \label{sec:Main-eq}
Delay equations arise naturally in mathematical ecology,
where changes in a population may depend on its past
rather than only its present state. Recruitment, for
example, may reflect a range of maturation times,
leading to a contribution distributed over a history
interval. A measure on this interval provides a unified
description: atoms represent discrete delays, while
densities describe continuously distributed contributions.
Allowing the measure to vary with time accommodates
changes in the strength and distribution of these effects.

We consider a linear formulation in which a finite signed
measure $\mu_t$ weights the contribution of the past
values $z(t+s)$ to the current rate of change.
The use of signed measures allows both positive and
negative contributions, as may occur in linearizations
of population models. With the maximal delay normalized
to one, we study the following Carath\'{e}odory type initial-value problem 
\begin{equation}
\label{eq:IVP-C-ODEs}
\begin{cases}
z'(t) = \displaystyle \int_{[-1,0]} z(t+s) \, \dd\mu_t (s), & t \in [0, T]
\\[2ex] 
z(t) = u_0(t), & t \in [-1, 0],
\end{cases}
\end{equation}
where the mapping $\mu$ is assumed to satisfy the following.     
\begin{enumerate}[label=(\textup{A\arabic*}),series=supo,leftmargin=24pt] 
\item\label{A2-ODEs} The mapping 
\begin{equation*}
    [0,T] \ni t\mapsto \langle \cdot ,\mu_t \rangle \in C^* 
\end{equation*}
is weak\nobreakdash-\hspace{0pt}* measurable i.e., for any test function $f\in C$ the mapping
\begin{equation*}
     [0,T] \ni t\mapsto \langle f ,\mu_t \rangle \in \RR
\end{equation*}
is $(\mathfrak{L}([0, T]), \mathfrak{B}(\RR))$\nobreakdash-\hspace{0pt}measurable.
\end{enumerate}

We define the space $\mathfrak{m}$ as a $\mathrm{TV}$-norm-bounded (by some fixed $K>0$) and closed subset of Borel measures defined on $\mathfrak{B}([-1,0])$ equipped with the relative weak-* topology. Namely, 
\begin{equation*}
    \mathfrak{m} = \{ \nu :   \text{ signed measure on } \mathfrak{B}([-1,0]) \text{ such that } \norm{\nu}_{\mathrm{TV}} \le K\}.
\end{equation*}
We have the classical result for the weak-* compactness of $\mathfrak{m}$. 
\begin{remark}
       $\mathfrak{m}$ is weak-* compact, and its weak-* topology is metrizable.
\end{remark}
\noindent
See~\cite[Thm. 3.6.17]{DM} for more details. In order to investigate continuous dependence on measure valued functions (coefficient paths) we define the coefficient space $\mathfrak{m}_T$ as a subset of $\mathfrak{m}^{ [0,T]}$ of mappings $\mu$ that satisfy assumption~\ref{A2-ODEs}.  In other words, we have
\begin{equation*}
    \mathfrak{m}_T \vcentcolon = \{ \, \mu \in \mathfrak{m}^{[0,T]} : 
    \ref{A2-ODEs}
    \,\}.
\end{equation*}
The space $\mathfrak m_T$ inherits the relative product
(Tychonoff) topology from $\mathfrak m^{[0,T]}$, where
$\mathfrak m$ carries the weak-* topology.
Convergence in this topology requires weak-* convergence
at every $t\in[0,T]$.

For the parameter-dependence results, we use the weaker
notion of convergence
\begin{equation*}
    \mu^{(n)}\to\mu
    \quad\Longleftrightarrow\quad
    \mu_t^{(n)}\overset{*}{\rightharpoonup}\mu_t
    \text{ for Lebesgue-almost every }t\in[0,T].
\end{equation*}
We regard this as a prescribed sequential convergence
structure, without assuming that it is induced by
a topology. Unless stated otherwise, convergence of
coefficient paths refers to this notion.
Since convergence in the relative product topology
implies the convergence above, all subsequent results
on sequential dependence also apply to sequences
converging in the Tychonoff topology. This does not,
however, assert continuity in that topology.

In some arguments below, the uniform total variation bound
can be replaced by the weaker condition
\begin{equation*}
    \int_0^T \|\mu_t\|_{\mathrm{TV}}\,\dd t<\infty,
\end{equation*}
with measurability of $[t\mapsto \|\mu_t\|_{\mathrm{TV}}]$ ensured by \ref{A2-ODEs}; see Lemma~\ref{lemma:mu-measurable}.
This is consistent with the integrable bounds used in
classical Carath\'eodory theory. In particular, the
existence and uniqueness arguments and the associated
Gr{\"o}nwall estimates can be adapted to this weaker assumption.
By contrast, our proofs of sequential continuous dependence
use the common uniform bound to obtain estimates and
domination independent of the coefficient path.
For consistency, we retain the uniform boundedness
assumption throughout and do not specify separately
which results admit weaker hypotheses.

\smallskip\par 
We begin with two auxiliary lemmas. See~\cite{MK-diss} for more details. 
\begin{lemma}
  For any $\mu\in\mathfrak{m}_T$ and $z\in C([-1,T])$ the mapping 
    \begin{equation}\label{eq:w^*-measurable-pairing}
       [0,T] \ni t\mapsto \langle z_t ,\mu_t \rangle  \in \RR 
    \end{equation}
    is in $L_{\infty}([0,T])$.
\end{lemma}  

\begin{proof} We start by showing the $(\mathfrak{L}([0, T]), \mathfrak{B}(\RR))$\nobreakdash-\hspace{0pt}measurability of the mapping~\eqref{eq:w^*-measurable-pairing}. Notice that it can be rewritten as a composition of 
\begin{equation*}
            t\mapsto (  z_t  , \mu_t ) \mapsto  \langle z_t ,\mu_t \rangle,   
\end{equation*}
where the first mapping is $(\mathfrak{L}([0, T]), \mathfrak{B}(C)\otimes \mathfrak{B}( C^*,w^*) )$\nobreakdash-\hspace{0pt}measurable (recall that $\mathfrak{B}( C^*,w^*)$ stands for the Borel $\sigma$\nobreakdash-\hspace{0pt}algebra generated by the \WST topology) and the second mapping is $(\mathfrak{B}(C \times (C^*,w^*)) ,  \mathfrak{B}(\RR) )$\nobreakdash-\hspace{0pt}measurable. Therefore, by~\cite[Lemma~6.4.2(i)]{V.B.} and the separability of $C$, the mapping~\eqref{eq:w^*-measurable-pairing} is $(\mathfrak{L}([0, T]),\allowbreak \mathfrak{B}(\RR))$-measurable. It remains to notice that the mapping~\eqref{eq:w^*-measurable-pairing} is bounded by $K\norm{z}_{C([-1,T])}$, so we obtain the membership in $L_{\infty}([0,T])$. 
\end{proof} 

\begin{lemma}\label{lemma:mu-measurable}
    For any $\mu\in\mathfrak{m}_T$ the mapping 
    \begin{equation*}
              [0,T] \ni t\mapsto \| \mu_t \|_{\mathrm{TV}}\in \RR  
    \end{equation*}
    is in $L_{\infty}([0,T])$. Similarly, the mappings 
        \begin{equation*}
         \big[\, [0,T] \ni t \mapsto  \mu_t^+([-1,0]) \in \RR \, \big  ] \quad\text{and}\quad  \big[\, [0,T] \ni t \mapsto \mu_t^-([-1,0]) \in \RR \, \big  ]. 
    \end{equation*}
    are in $L_{\infty}([0,T])$, too.
\end{lemma}
\begin{proof}
    The boundedness of those mappings is clear since $\|\mu_t\|_{\mathrm{TV}} \le K$ for all $t\in [0,T]$. It suffices to notice that from weak\nobreakdash-\hspace{0pt}* measurability it follows that the mapping  
    \begin{equation*}
           [0,T] \ni t \mapsto \int_{[-1,0]} \dd \mu_t =  \mu_t([-1,0]) \in \RR 
    \end{equation*}
    is $(\mathfrak{L}([0,T]) ,  \mathfrak{B}(\RR) )$\nobreakdash-\hspace{0pt}measurable and further, 
    \begin{equation*}
        \| \mu_t \|_{\mathrm{TV}}  = \sup\Big\{ \Big|\int_{[-1,0]} f(s) \, \dd \mu_t (s)\Big| : f\in \mathcal{G} \Big\}, 
    \end{equation*}
    where $\mathcal{G}$ is a countable, dense subset of the closed unit ball in $C([-1,0])$. 
    Moreover, maps 
    \begin{equation*}
         \big[\, [0,T] \ni t \mapsto  \mu_t^{\pm}([-1,0]) \in \RR \, \big  ] 
    \end{equation*}
are also  $(\mathfrak{L}([0,T]) ,  \mathfrak{B}(\RR) )$\nobreakdash-\hspace{0pt}measurable since we have 
\begin{equation*}
      \mu_t^{\pm}([-1,0]) =\frac{  \| \mu_t \|_{\mathrm{TV}} \pm  \mu_t([-1,0])}{2}. 
\end{equation*}
\end{proof}
For completeness, we present two classical results concerning continuity and sequential continuity of a contraction mapping with respect to parameters. Let $x^s$ be the unique fixed point of a mapping $F(\cdot,s):X\to X$ parametrised by $S$, where $S$ is a topological space (in general not necessarily metrizable). The next two theorems, also known as the uniform contraction principle, provide sufficient conditions under which the mapping $[\, S\ni s\mapsto x^s \in X \,]$ is continuous. See~\cite[Thm. B.2]{Aulbach}. 
\begin{theorem}
\label{thm:continuous_dependence_fixed_point}
    Let $(X,d)$ be a complete metric space, $S$ a topological space and $F \colon X \times S\to X$ such that 
 \begin{itemize}
     \item for each $x\in X$ the mapping  $F(x,\cdot ) \colon S \to X$ is continuous, 
     \item there exists $\varkappa\in [0,1)$  such that  for each $s\in S$ the mapping $F(\cdot, s) \colon X \to X$  is $\varkappa$\nobreakdash-\hspace{0pt}Lipschitz, i.e., 
     \begin{equation*}
        d(F(x,s), F(y,s))\le \varkappa d(x, y),
    \end{equation*}
    for all $x,y\in X$ and $s\in S.$
 \end{itemize}   
 Then the unique fixed point obtained from the Banach fixed point theorem depends continuously on $s\in S$. 
\end{theorem}

\begin{theorem}
\label{thm:seq_continuous_dependence_fixed_point}
    Let $(X,d)$ be a complete metric space, $S$ a topological space and $F \colon X \times S \to X$ such that 
 \begin{itemize}
     \item for each $x\in X$ the mapping  $F(x,\cdot ) \colon S\to X$ is sequentially continuous, 
     \item there exists $\varkappa\in [0,1)$  such that  for each $s\in S$ the mapping $F(\cdot, s) \colon X \to X$ is $\varkappa$\nobreakdash-\hspace{0pt}Lipschitz. 
 \end{itemize}   
 Then the unique fixed point obtained from the Banach fixed point theorem depends sequentially continuously on $s\in S$. 
\end{theorem}

\section{Carath\'{e}odory type solution formulation and continuous dependence on parameters}\label{sec:Carath-sol}

We first establish existence and uniqueness through
the Carath\'eodory integral formulation.
Weighted norms make the continuation operator contractive
on an arbitrary finite time interval.
We then study dependence on the measure-valued parameter,
distinguishing almost-everywhere weak-* convergence
from continuity in the relative product topology.

\begin{definition}
   For fixed parameters $\mu\in \mathfrak{m}_T$ and $u_0 \in C([-1, 0])$, a solution of the equation~\eqref{eq:IVP-C-ODEs} is defined as a continuous function $z\in C([-1,T])$ satisfying the integral equation 
   \begin{equation}\label{eq:int-variant}
       z(t) = u_0(0) + \int_0^t \displaystyle \int_{[-1,0]} z(\tau +s) \, \dd\mu_{\tau} (s) \,\dd \tau  
   \end{equation}
   for $t\ge 0$ and $z(t)=u_0(t)$ for $t\in[-1,0]$. 
\end{definition}
\begin{proposition} \label{prop:unigue-sol-in-C}
   For any $\mu\in \mathfrak{m}_T$ and $u_0 \in C([-1, 0])$, the equation~\eqref{eq:IVP-C-ODEs} has a unique Carath\'{e}odory type solution. 
\end{proposition}
 The proof is based on a contraction argument.  The solution is obtained as a fixed point of the contraction mapping $\mathfrak{G}_{\mu}$ of $\{z\in C([0,T]) : z(0)=u_0(0)\}\vcentcolon ={C_{u_0}([0,T])}$ into itself defined as
\begin{equation}
\label{eq:Gothic-G-ODEs}
  \begin{aligned}
  \mathfrak{G}_{\mu}(z)[t] \vcentcolon= {} u_0(0)+\int_0^t \int_{[-1,0]} z(\tau+s)  \, \dd \mu_{\tau}(s) \,\dd\tau,
  \end{aligned}
\end{equation}
where $0 \le t \le T$.
The operator depends on the entire prescribed history $u_0$,
whereas the space $C_{u_0}([0,T])$ depends only on its value
at zero. Throughout arguments in which $u_0$ is fixed,
we write $\mathfrak G_{\mu}$ instead of
$\mathfrak G_{\mu,u_0}$.
Whenever a continuation $z$ is evaluated at a negative
argument, it is understood to be extended by $u_0$ on
$[-1,0]$. We use the same convention for
$z\in AC_{u_0}([0,T])$.
Until revoking, $z$ stands for a generic function in $C$. Recall the metric $d_{\kappa}$ on $C([0, T])$ by the formula
\begin{equation*}
    d_{\kappa}(u, v) \vcentcolon = \sup{\{\, e^{-{\kappa}t} |u(t) - v(t)| : t \in [0, T] \,\}}, \quad u, v \in C([0,T]),
\end{equation*}
where $\kappa > 0$ is a positive number (to be chosen later).  The metric $d_{\kappa}$ is complete and equivalent to the metric induced by the usual norm on $C([0,T])$.

\begin{proposition}
\label{prop:contraction-ODEs}
There exists $\kappa > 0$ such that for any $\mu\in \mathfrak{m}_T$, $u_0 \in C([-1, 0])$ the operator $\mathfrak{G}_{\mu}$ defined on the closed set $C_{u_0}([0,T])$ into itself is a contraction map in the $d_{\kappa}$ metric, with the contraction coefficient independent of $\mu$ and $u_0$.
\end{proposition}
\begin{proof}
    \begin{equation*}
        \begin{split}
             |\mathfrak{G}_{\mu}(z)[t]   -  \mathfrak{G}_{\mu}(w)[t] | &  \le  \int_{0}^{t}  \int_{[-1,0]} | z(\tau+s) -  w(\tau+s) | \, \dd |\mu_{\tau}|(s)  \, \dd\tau  \\
    & \le   \int_{0}^{t}  \int_{[-1,0]} e^{\kappa \tau} d_{\kappa}(z,w)  \, \dd |\mu_{\tau}|(s) \,\dd\tau \\
    & \le  K  d_{\kappa}(z,w) \int_{0}^{t}  e^{\kappa \tau} \,\dd\tau  \\
    & = K \frac{(e^{\kappa t}-1)}{\kappa} d_{\kappa}(z,w). 
        \end{split}
    \end{equation*}
Therefore, 
\begin{align*}  d_{\kappa}(\mathfrak{G}_{\mu}(z) , \mathfrak{G}_{\mu}(w)) & = \sup{\{\, e^{-{\kappa}t}  |\mathfrak{G}_{\mu}(z)[t]   - \mathfrak{G}_{\mu}(w)[t] | : t \in [0, T] \,\}} \\
        &\le \sup{\Big\{\,  K\frac{(1-e^{-{\kappa}t})}{\kappa} : t \in [0, T] \,\Big\}} \, d_{\kappa}(z,w)\\
        &\le \frac{K}{\kappa} d_{\kappa}(z,w). 
\end{align*}
For sufficiently large $\kappa>K$ the Contraction Mapping Principle guarantees the existence and uniqueness of the fixed point $z$ of $\mathfrak{G}_{\mu}$, which is then, after extending to $[-1, 0]$ by $u_0$, the unique  Carath\'{e}odory type solution of~\eqref{eq:IVP-C-ODEs} on $[-1, T]$ satisfying the initial condition.
\end{proof}

\begin{proposition}
\label{prop:G_operator_continuous}
   Let $u_0\in C([-1, 0])$. We have the following.
   \begin{enumerate}[label=\textup{(}\roman*\textup{)}, ref=\textup{(}\emph{\roman*}\textup{)}]
   \item\label{prop:G_operator_continuous:i}
   For any $\mu\in\mathfrak{m}_T$ the mapping
   \begin{equation*}
           C_{u_0}([0,T]) \ni z
           \mapsto \mathfrak{G}_{\mu}(z)
           \in C_{u_0}([0,T])
   \end{equation*}
   is continuous.

   \item\label{prop:G_operator_continuous:ii}
   For any $z\in C_{u_0}([0,T])$ the mapping
   \begin{equation*}
           \mathfrak{m}_T \ni \mu
           \mapsto \mathfrak{G}_{\mu}(z)
           \in C_{u_0}([0,T])
   \end{equation*}
   is sequentially continuous.

   \item\label{prop:G_operator_continuous:iii}
   \begin{equation*}
           C_{u_0}([0,T]) \times \mathfrak{m}_T
           \ni (z,\mu)
           \mapsto \mathfrak{G}_{\mu}(z)
           \in C_{u_0}([0,T])
   \end{equation*}
   is jointly sequentially continuous.
   \end{enumerate}
\end{proposition}
\begin{proof}
    It suffices to show only the last part (i.e.,~\ref{prop:G_operator_continuous:iii}). Fix $z\in C_{u_0}([0,T])$ and $\mu \in  \mathfrak{m}_T$. Let  $(z_m)_{m=1}^{\infty}\subset C_{u_0}([0,T])$ and $(\mu^{(m)})_{m=1}^{\infty}\subset \mathfrak{m}_T$ converge to $z$ and $\mu$, respectively. We can estimate the norm of the difference as follows
\begin{multline*}
    \|\mathfrak{G}_{\mu^{(m)}}(z_m) - \mathfrak{G}_\mu(z)\|_{C([0,T])} \\ \le \|\mathfrak{G}_{\mu^{(m)}}(z_m) - \mathfrak{G}_{\mu^{(m)}}(z)\|_{C([0,T])} + \|\mathfrak{G}_{\mu^{(m)}}(z) - \mathfrak{G}_\mu(z)\|_{C([0,T])}.
\end{multline*}
Let us estimate the first term on the right-hand side. By the definition of the operator $\mathfrak{G}$ and the fact that the sequence of measures $(\mu^{(m)})_{m=1}^\infty$ is uniformly bounded in $\mathfrak{m}_T$ we obtain 
\begin{equation*}
\begin{split}
    \|\mathfrak{G}_{\mu^{(m)}}(z_m) -   \mathfrak{G}_{\mu^{(m)}}(z)&\|_{C([0,T])}  \\ & =  \sup_{t \in [0,T]} \Big| \int_0^t \int_{[-1,0]} (z_m(\tau+s) - z(\tau+s)) \, \dd \mu^{(m)}_{\tau}(s) \, \dd \tau \Big| \\
&\le \sup_{t \in [0,T]} \int_0^t \|\mu^{(m)}_{\tau}\| \sup_{\xi \in [\tau-1, \tau]} |z_m(\xi) - z(\xi)| \,\dd\tau \\
&\le KT  \|z_m - z\|_{C([0,T])}.
\end{split}
\end{equation*}
Since $z_m \to z$ in $C_{u_0}([0,T])$, we have $\|z_m - z\|_{C([0,T])} \to 0$ as $m \to \infty$. Consequently, the first term converges to zero. Convergence of the second term follows from the Lebesgue Dominated Convergence Theorem. We have 
\begin{equation*}
    \begin{split}
        \|\mathfrak{G}_{\mu^{(m)}}(z) - \mathfrak{G}_\mu(z)\|_{C([0,T])} &\le \sup_{t \in [0,T]} \int_0^t \Big| \int_{[-1,0]} z(\tau+s) \, \dd(\mu^{(m)}_{\tau} - \mu_\tau)(s) \Big| \,\dd\tau \\
&\le \int_0^T \Big| \int_{[-1,0]} z(\tau+s) \,\dd(\mu^{(m)}_{\tau} - \mu_\tau)(s) \Big|\, \dd\tau.
    \end{split}
\end{equation*}
Since $\mu^{(m)} \to \mu$ (i.e., $\mu^{(m)}_{\tau} \to \mu_{\tau}$ in *-weak topology for almost every $\tau \in [0,T]$) the inner integral converges to zero for almost every $\tau \in [0,T]$, meaning
\begin{equation*}
    \int_{[-1,0]} z(\tau+s)\, \dd(\mu^{(m)}_{\tau} - \mu_\tau)(s)  \to 0 \quad \text{for a.e. } \tau \in [0,T].
\end{equation*}
Moreover, the sequence of measures is uniformly bounded, so we have 
\begin{equation*}
    \Big|\int_{[-1,0]} z(\tau+s)\, \dd(\mu^{(m)}_{\tau} - \mu_\tau)(s)  \Big|  \le 2K\max\{\|u_0\|_{C([-1,0])},\|z\|_{C([0,T])}\}.
\end{equation*}
The bounding constant $2K\max\{\|u_0\|_{C([-1,0])},\|z\|_{C([0,T])}\}$ is trivially integrable over the bounded interval $[0,T]$. Therefore, we can apply the Lebesgue Dominated Convergence Theorem to conclude that:
\begin{equation*}
     \|\mathfrak{G}_{\mu^{(m)}}(z) - \mathfrak{G}_\mu(z)\|_{C([0,T])} \to 0
\end{equation*}
which finishes the second part and simultaneously completes the proof.
\end{proof}

\begin{proposition}
    \label{prop:G_operator_continuous-AC}
   For any $u_0\in C([-1, 0])$
       \begin{equation*}
          C_{u_0}([0,T]) \times \mathfrak{m}_T \ni (z, \mu)  \mapsto \mathfrak{G}_{\mu}(z) \in AC_{u_0}([0,T]) 
    \end{equation*}
    is jointly sequentially continuous. 
\end{proposition}
\begin{proof}
 Fix $z\in C_{u_0}([0,T])$ and $\mu \in  \mathfrak{m}_T$. Let  $(z_m)_{m=1}^{\infty}\subset C_{u_0}([0,T])$ and $(\mu^{(m)})_{m=1}^{\infty}\subset \mathfrak{m}_T$ converge to $z$ and $\mu$, respectively. We can estimate the norm of the difference as follows
  \begin{equation*}
\begin{split}
        \|\mathfrak{G}_{\mu^{(m)}}(z_m) - & \mathfrak{G}_\mu(z)\|_{AC([0,T])}   \\ 
        & \le \|\mathfrak{G}_{\mu^{(m)}}(z_m) - \mathfrak{G}_{\mu^{(m)}}(z)\|_{AC([0,T])} + \|\mathfrak{G}_{\mu^{(m)}}(z) - \mathfrak{G}_\mu(z)\|_{AC([0,T])}.
\end{split}
\end{equation*}
Proceeding similarly to the proof of Proposition~\ref{prop:G_operator_continuous}, we have
\begin{equation*}
\begin{split}
        \|\mathfrak{G}_{\mu^{(m)}}(z_m) - \mathfrak{G}_{\mu^{(m)}}(z)\|_{AC([0,T])} & = \int_0^T \Big| \int_{[-1,0]} (z_m(t+s) - z(t+s)) \,\dd\mu^{(m)}_{t}(s) \Big| \, \dd t \\ 
        & \le  KT  \|z_m - z\|_C.
\end{split}
\end{equation*}
Since $z_m \to z$ in $C$ we have $\|z_m - z\|_C \to 0$. Hence, the first term converges to zero. The second term is exactly the integral we analyzed previously (see Proposition~\ref{prop:G_operator_continuous}). 
\end{proof}

As a consequence of Proposition~\ref{prop:contraction-ODEs}, Proposition~\ref{prop:G_operator_continuous} and Theorem~\ref{thm:seq_continuous_dependence_fixed_point} we have the following.   
\begin{theorem}\label{thrm:C-cnontin-sol} For any $u_0\in C$ the mapping 
    \begin{equation*} 
         \mathfrak{m}_T \ni \mu \mapsto z(\cdot,\mu,u_0)\in C([-1, T])  
    \end{equation*}
    is sequentially continuous. 
\end{theorem}

Moreover, the integral identity~\eqref{eq:int-variant}
shows that the unique solution $z(\cdot,\mu,u_0)$ is
absolutely continuous on $[0,T]$.
However, Theorem~\ref{thrm:C-cnontin-sol} establishes
sequential continuity with respect to the parameters
only in the supremum norm. This does not automatically
imply sequential continuity in the stronger
$AC([0,T])$ norm, which also requires convergence
of the derivatives in $L_1([0,T])$.
The following results establish this stronger
parameter dependence.

\begin{theorem}\label{Thrm:familie-uniform-banach-uniform}
Let $(X,d)$ be a nonempty complete metric space and let $S$
be a topological space. Suppose that $(d_s)_{s\in S}$
is a family of metrics on $X$ satisfying
\begin{equation*}
    c\,d(x,y)\leq d_s(x,y)\leq C\,d(x,y),
    \qquad x,y\in X,\quad s\in S,
\end{equation*}
for constants $c,C>0$ independent of $s$.
Let $F:X\times S\to X$ satisfy
\begin{equation*}
    d_s(F(x,s),F(y,s))
    \leq \varkappa\,d_s(x,y),
    \qquad x,y\in X,\quad s\in S,
\end{equation*}
for some $\varkappa\in[0,1)$. Then each $F(\cdot,s)$ has a unique fixed point $x^s$.
If $F(x,\cdot):S\to(X,d)$ is continuous for every $x\in X$,
then $s\mapsto x^s$ is continuous.
The same holds with continuity replaced
throughout by sequential continuity.
\end{theorem}

\begin{proof}
The metric inequalities imply that $(X,d_s)$ is complete
for every $s\in S$. Hence the Banach fixed-point theorem
provides a unique fixed point $x^s\in X$ satisfying $F(x^s,s)=x^s$.  Fix $s_0\in S$. By the fixed-point identities and the
contraction estimate,
\begin{equation*}
    \begin{aligned}
        d_s(x^s,x^{s_0})
        &\leq
        d_s(F(x^s,s),F(x^{s_0},s))
        +
        d_s(F(x^{s_0},s),F(x^{s_0},s_0))\\
        &\leq
        \varkappa\,d_s(x^s,x^{s_0})
        +
        C\,d(F(x^{s_0},s),F(x^{s_0},s_0)).
    \end{aligned}
\end{equation*}
Rearranging and using $c\,d\leq d_s$, we obtain
\begin{equation*}
    d(x^s,x^{s_0})
    \leq
    \frac{C}{c(1-\kappa)}
    d(F(x^{s_0},s),F(x^{s_0},s_0)).
\end{equation*}
Let $\varepsilon>0$. By continuity of
$F(x^{s_0},\cdot)$, there exists an open neighbourhood
$V$ of $s_0$ such that
\begin{equation*}
    d(F(x^{s_0},s),F(x^{s_0},s_0))
    <
    \frac{c(1-\kappa)}{C}\,\varepsilon,
    \qquad s\in V.
\end{equation*}
The preceding estimate therefore gives $d(x^s,x^{s_0})<\varepsilon$ for $s\in V$. 
Thus, the fixed-point map is continuous at $s_0$.
Since $s_0$ was arbitrary, it is continuous on $S$.
\end{proof}

For fixed $\mu\in\mathfrak{m}_T$ we will assume that $AC([0,T])$ is equipped with the Sobolev--Bielecki type norm defined as 
\begin{equation*}
    \| x \|_{\kappa,\mu, AC([0,T])} := |x(0)| + \int_{0}^T |x'(t)| \exp\Big(-\kappa \int_0^t \|\mu_s\|_{\mathrm{TV}}\,\dd s\Big)\,  \dd t 
\end{equation*}
where as before $\kappa > 0$ is a positive number (to be chosen later). Since we assume $\mu\in \mathfrak{m}_T$ we have 
\begin{equation*}
    \exp(-\kappa K T)\le \exp\Big(-\kappa \int_0^t \|\mu_s\|_{\mathrm{TV}}\,\dd s \Big)\le 1 
\end{equation*}
for every $t\in[0,T]$. Therefore, the norm $\| x \|_{\kappa,\mu, AC([0,T])}$ is equivalent to the standard one. 
\begin{proposition}
\label{prop:contraction-ODEs-on-AC}
There exists $\kappa > 0$ such that for any $\mu\in \mathfrak{m}_T$, $u_0 \in C([-1, 0])$ the operator $\mathfrak{G}_{\mu}$ defined on the closed set $AC_{u_0}([0,T])$ into itself is a contraction map in the $\|\cdot\|_{\kappa, \mu, AC}$ norm, with the contraction coefficient independent of $\mu$ and $u_0$.
\end{proposition}
\begin{proof}
    Fix $\mu$ and $u_0$ as in the statement. We start from the pointwise estimation 
    \begin{equation*}
        \begin{split}
            |(\mathfrak{G}_{\mu}z)'[t] - (\mathfrak{G}_{\mu}w)'[t]| & \le \Big| \int_{[-1,0]} z(t+s) \, \dd \mu_t(s)- \int_{[-1,0]} w(t+s) \, \dd \mu_t(s) \Big|  \\
            & \le \int_{[-1,0]} \int_0^{t+s} \vert{}z'(r) - w'(r)\vert{} \, \dd r \, \dd |\mu_t|(s) \\
            & \le \|\mu_t\|_{\mathrm{TV}} \int_0^{t} \vert{}z'(r) - w'(r)\vert{} \, \dd r
        \end{split}
    \end{equation*}
and 
\begin{multline*}
        \int_r^T \Vert{}\mu_t\Vert{}_{\text{TV}} \exp\Big(-\kappa \int_0^t \Vert{}\mu_s\Vert{}_{\text{TV}} \, \text{d}s\Big) \, \text{d}t \\ =  \frac{1}{\kappa} \exp\Big(-\kappa \int_0^r \Vert{}\mu_s\Vert{}_{\text{TV}} \, \text{d}s\Big) - \frac{1}{\kappa} \exp\Big(-\kappa \int_0^T \Vert{}\mu_s\Vert{}_{\text{TV}} \, \text{d}s\Big) \\ \le \frac{1}{\kappa} \exp\Big(-\kappa \int_0^r \Vert{}\mu_s\Vert{}_{\text{TV}} \, \text{d}s\Big). 
\end{multline*}
Therefore, via the Fubini theorem, we have 
\begin{equation*}
\begin{split}
     \Vert{}\mathfrak{G}_\mu  (z) & - \mathfrak{G}_\mu (w)\Vert{}_{\kappa, \mu, AC}  = \int_0^T \vert{}(\mathfrak{G}_\mu z)'[t] - (\mathfrak{G}_\mu w)'[t]\vert{} \exp\Big(-\kappa \int_0^t \Vert{}\mu_s\Vert{}_{\text{TV}} \, \text{d}s\Big) \, \text{d}t \\
     & \le \int_0^T \Big( \Vert{}\mu_t\Vert{}_{\text{TV}} \int_0^t \vert{}z'(r) - w'(r)\vert{} \, \text{d}r \Big) \exp\Big(-\kappa \int_0^t \Vert{}\mu_s\Vert{}_{\text{TV}} \, \text{d}s\Big) \, \text{d}t \\
     & = \int_0^T \vert{}z'(r) - w'(r)\vert{} \Big(  \int_r^T \Vert{}\mu_t\Vert{}_{\mathrm{TV}} \exp\Big(-\kappa \int_0^t \Vert{}\mu_s\Vert{}_{\mathrm{TV}} \, \dd s\Big) \, \dd t \Big) \, \dd r \\
     & \le \frac{1}{\kappa} \int_0^T \vert{}z'(r) - w'(r)\vert{} \exp\Big(-\kappa \int_0^r \Vert{}\mu_s\Vert{}_{\text{TV}} \, \text{d}s\Big) \, \text{d}r = \frac{1}{\kappa} \Vert{}z - w\Vert{}_{\kappa, \mu, AC}. 
\end{split}
\end{equation*}  
\end{proof}

\begin{theorem} For every fixed $u_0\in C([-1,0])$, the mapping
\begin{equation*}
    \mathfrak{m}_T\ni\mu
    \mapsto z(\cdot;\mu,u_0){\restriction}_{[0,T]}
    \in AC([0,T])
\end{equation*}
is sequentially continuous.
If, in addition, $u_0\in AC([-1,0])$, the same conclusion
holds for the full solution with values in $AC([-1,T])$.
\end{theorem}
\begin{proof}
The assumptions of Theorem~\ref{Thrm:familie-uniform-banach-uniform} for $(\| \cdot \|_{\kappa,\mu, AC})_{\mu\in \mathfrak{m}_T}$ are verified in 
    Proposition~\ref{prop:G_operator_continuous-AC}  and   Proposition~\ref{prop:contraction-ODEs-on-AC}. 
\end{proof}

The following result distinguishes sequential continuity from
continuity with respect to the product topology on the parameter
space. Although pointwise weak-* convergence of sequences is
sufficient for continuity of the integral operator along sequences,
the operator need not be continuous in the topological sense.

\begin{theorem}
For every fixed $x\in C([-1,T])$, the map
\begin{equation*}
    \mathfrak m_T\ni\mu
    \mapsto \mathfrak{G}_\mu (x)\in AC([0,T])
\end{equation*}
is sequentially continuous. However, this map need not be
continuous: for $x\equiv 1$, it is discontinuous at the zero
parameter.
\end{theorem}

\begin{proof}
We prove discontinuity by exhibiting an open set whose inverse image is not open. Fix $x\equiv 1$ and write
\begin{equation*}
    \Phi(\mu):=\mathfrak{G}_\mu 1.
\end{equation*}
Then $\Phi(0)\equiv 1$. Consider the open ball
\begin{equation*}
    V
    :=
    \big\{
        v\in AC([0,T]):
        \|v-1\|_{AC}<KT/2
    \big\}.
\end{equation*}
Clearly, $0\in\Phi^{-1}(V)$. Let $W$ be any neighbourhood of $0$ in $\mathfrak m_T$.
By the definition of the relative product topology, there
exist finitely many times $t_1,\ldots,t_m\in[0,T]$
and weak-* open neighbourhoods $O_1,\ldots,O_m$ of the
zero measure in $\mathfrak m$ such that
\begin{equation*}
    0\in
    \{
        \mu\in\mathfrak m_T:
        \mu_{t_j}\in O_j,\quad j=1,\ldots,m
    \}
    \subseteq W.
\end{equation*}
Define $ \widehat{\mu}=( \widehat{\mu}_t)_{t\in [0,T]}$ such that 
\begin{equation*}
    \widehat{\mu}_t
    :=
    \begin{cases}
        0,
        & t\in\{t_1,\ldots,t_m\},\\
        K\delta_0,
        & t\in[0,T]\setminus\{t_1,\ldots,t_m\}.
    \end{cases}
\end{equation*}
Then $\widehat{\mu}\in\mathfrak m_T$. Since $\widehat{\mu}_{t_j}=0\in O_j$
for every $j$, we have $\widehat{\mu}\in W$. On the other hand, changing an integrand at finitely many
times does not affect its Lebesgue integral. Thus,
\begin{equation*}
    \Phi(\widehat{\mu})(t)
    =
    1+\int_0^t
    \widehat{\mu}_\tau([-1,0])\,\dd\tau
    =
    1+Kt,
    \qquad t\in[0,T].
\end{equation*}
Consequently,
\begin{equation*}
    \|\Phi(\widehat{\mu})-1\|_{AC}
    =
    \int_0^T K\,\dd t
    =
    KT,
\end{equation*}
so $\widehat{\mu}\notin\Phi^{-1}(V)$. Hence no neighbourhood of $0$ is contained in
$\Phi^{-1}(V)$, even though $0\in\Phi^{-1}(V)$.
The inverse image $\Phi^{-1}(V)$ is therefore not open,
and $\Phi$ is discontinuous at $0$.
\end{proof}

\subsection{ Gr{\"o}nwall estimates}
The classical Gr{\"o}nwall  inequality also yields bounds
for solutions of the delay equation considered here.
For completeness, we derive estimates in the supremum
and absolutely continuous norms, which will subsequently
be used to control the solution operators.

\begin{proposition}\label{prop:Gronwal-determine-bound}
Let $\mu\in\mathfrak m_T$ and $u_0\in C([-1,0])$.
Then the solution $z=z(\cdot;\mu,u_0)$ satisfies
\begin{equation*}
    \sup_{-1\leq r\leq t}|z(r)|
    \leq e^{Kt}\|u_0\|_{C([-1,0])},
    \qquad t\in[0,T],
\end{equation*}
and
\begin{equation*}
    \|z{\restriction}_{[0,T]}\|_{AC([0,T])}
    \leq e^{KT}\|u_0\|_C.
\end{equation*}
If, in addition, $u_0\in AC([-1,0])$, then
$z\in AC([-1,T])$ and
\begin{equation*}
    \|z\|_{AC([-1,T])}
    \leq e^{KT}\|u_0\|_{AC([-1,0])}.
\end{equation*}
\end{proposition}

\begin{proof}
For $0\leq r\leq t\leq T$, the integral equation gives
\begin{equation*}
    \begin{aligned}
        |z(r)|
        &\leq
        |u_0(0)|
        +
        \int_0^r\int_{[-1,0]}
        |z(\tau+s)|\,\dd|\mu_\tau|(s)\,\dd\tau\\
        &\leq
        \|u_0\|_C
        +
        K\int_0^t
        \sup_{-1\leq \xi\leq\tau}|z(\xi)|\,\dd\tau.
    \end{aligned}
\end{equation*}
For $r\in[-1,0]$, we have $|z(r)|\leq\|u_0\|_C$.
Taking the supremum over $r\in[-1,t]$, we obtain
\begin{equation*}
    \sup_{-1\leq r\leq t}|z(r)|
    \leq
    \|u_0\|_C
    +
    K\int_0^t
    \sup_{-1\leq \xi\leq\tau}|z(\xi)|\,\dd\tau.
\end{equation*}
The running supremum is continuous in $t$, since $z$
is continuous. Gr{\"o}nwall 's inequality therefore yields
\begin{equation*}
    \sup_{-1\leq r\leq t}|z(r)|
    \leq e^{Kt}\|u_0\|_C.
\end{equation*}
Using the differential equation, we consequently obtain
for almost every $t\in[0,T]$
\begin{equation*}
    |z'(t)|
    \leq
    \|\mu_t\|_{\mathrm{TV}}
    \sup_{-1\leq r\leq t}|z(r)|
    \leq Ke^{Kt}\|u_0\|_C.
\end{equation*}
Hence
\begin{equation*}
    \begin{aligned}
        \|z{\restriction}_{[0,T]}\|_{AC([0,T])}
        &=
        |u_0(0)|+\int_0^T|z'(t)|\,\dd t\\
        &\leq
        \|u_0\|_C
        +
        \|u_0\|_C\int_0^T Ke^{Kt}\,\dd t\\
        &=e^{KT}\|u_0\|_C.
    \end{aligned}
\end{equation*}
Finally, if $u_0\in AC$, the two absolutely continuous
parts agree at zero, so $z\in AC([-1,T])$.
Moreover,
\begin{equation*}
    \begin{aligned}
        \|z\|_{AC([-1,T])}
        &=
        \|u_0\|_{AC}+\int_0^T|z'(t)|\,\dd t\\
        &\leq
        \|u_0\|_{AC}+(e^{KT}-1)\|u_0\|_C\\
        &\leq e^{KT}\|u_0\|_{AC}.
    \end{aligned}
\end{equation*}
\end{proof}

\section{Measure integration}\label{sec:Measure-integration-and-disintegration}
In this section we construct a signed measure on the time–delay domain from the family $(\mu_t)_{t\in[0,T]}$. To be more specific for a given $t\in[0,T]$ we define $\nu_t$ by the formula 
\begin{equation*}
    \nu_t(A)= \int_{0}^t \mu_{\tau}(A_{\tau})\, \dd \tau, 
\end{equation*}
where $A\in\mathfrak{B}([0,t]\times [-1,0])$  and $\hypertarget{A-tau-sect}{A_{\tau}}\vcentcolon =\{s\in [-1,0]: (\tau,s)\in A\}$. We will show that $\nu_t$ is a signed measure, and moreover, 
\begin{equation}\label{eq:measure-integration-eq}
  \int_0^t \int_{[-1,0]} f(\tau,s) \,\dd \mu_{\tau}(s) \,\dd\tau = \int_{[0,t]\times [-1,0]} f(\tau,s) \, \dd \nu_t(\tau,s)  
\end{equation} 
holds for $f\in L_1(\nu_t)$.    

Before we start we recall a classical result known as monotone class theorem, since it will be used later on. 

\begin{theorem}[Monotone class theorem]
\label{thm:monotone-class}
Let $X$ be a set and let $\mathcal A$ be an algebra
of subsets of $X$. Thus, $X\in\mathcal A$, and
$\mathcal A$ is closed under complements relative to $X$
and finite unions.

A family $\mathcal M\subseteq\mathcal P(X)$ is called
a monotone class if it is closed under increasing
countable unions and decreasing countable intersections:
\begin{equation*}
    A_n\in\mathcal M,\quad A_n\subseteq A_{n+1}
    \quad\Longrightarrow\quad
    \bigcup_{n=1}^{\infty}A_n\in\mathcal M,
\end{equation*}
and
\begin{equation*}
    A_n\in\mathcal M,\quad A_{n+1}\subseteq A_n
    \quad\Longrightarrow\quad
    \bigcap_{n=1}^{\infty}A_n\in\mathcal M.
\end{equation*}
Then the smallest monotone class containing $\mathcal A$
is precisely $\sigma(\mathcal A)$. Consequently, if $\mathcal M$ is a monotone class
containing $\mathcal A$, then $\sigma(\mathcal A)\subseteq\mathcal M$. 
\end{theorem}

\smallskip\par
\begin{lemma}\label{lemma:borel-monotone-class-measurable} For any $A\in    \mathfrak{B}([0,T]\times [-1,0])$ the mapping 
\begin{equation}
     \big[\,[0,T]\ni t\mapsto \mu_t(A_{t})\in\RR\,\big] 
\end{equation}
is $(\mathfrak{L}([0, T]), \mathfrak{B}(\RR))$\nobreakdash-\hspace{0pt}measurable
\end{lemma}
\begin{proof}
 The proof is based on the monotone class theorem. We start by showing that the set 
 \begin{equation*}
 \begin{split}
          \mathcal{M} : = \big\{\,  A\in \mathfrak{B}([0,T]\times  [-1,0]) :  \big[\,[0,T]\ni & \tau\mapsto \mu_{\tau}(A_{\tau})\in\RR\,\big]   \\ & \text{ is } (\mathfrak{L}([0, T]), \mathfrak{B}(\RR))\text{\nobreakdash-\hspace{0pt}measurable} \, \big\}.  
 \end{split}
 \end{equation*}
 is a monotone class. Let $(A_m)_{m=1}^{\infty}\subset  \mathcal{M}$ such that $A_1\subset A_2\subset \dots $. Since, for any $\tau$ we have  
   \begin{equation*}
            \mu_{\tau}\Big(\Big(\bigcup_{m=1}^{\infty} A_{m}\Big)_{\tau}\Big)  =  \mu_{\tau}\Big(\bigcup_{m=1}^{\infty} (A_{m})_{\tau}\Big)  
       =\lim_{m\to\infty} \mu_{\tau}( (A_{m})_{\tau}), 
   \end{equation*}
  therefore, the mapping 
   \begin{equation*}
         \Big[\,[0,T]\ni \tau \mapsto \mu_{\tau}\Big(\Big(\bigcup_{m=1}^{\infty} A_{m}\Big)_{\tau}\Big) \in\RR \,\Big]
   \end{equation*}
   is measurable as a point-wise limit of measurable mappings and $A_1\cup A_2\cup\dots \in \mathcal{M}$. Similarly for any $(B_n)_{m=1}^{\infty}$  such that $B_1\supset B_2\supset \dots$ we have 
      \begin{equation*}
        \mu_{\tau}\Big(\Big(\bigcap_{m=1}^{\infty} B_{m}\Big)_{\tau}\Big)  =  \mu_{\tau}\Big(\bigcap_{m=1}^{\infty} (B_{m})_{\tau}\Big) 
       =\lim_{n\to\infty} \mu_{\tau}( (B_{n})_{\tau}). 
      \end{equation*}
   As before, the mapping 
    \begin{equation*}
         \Big[\,[0,T] \ni \tau \mapsto \mu_{\tau}\Big(\Big(\bigcap_{n=1}^{\infty} B_{n}\Big)_{\tau}\Big) \in\RR \,\Big]
   \end{equation*}
   is measurable as a pointwise limit of measurable maps. Hence, $B_1\cap B_2\cap\dots \in \mathcal{M}$. We are going to show now that monotone class contains the ring (algebra of sets) 
\begin{equation*}
\begin{split}
    \mathcal R :=
    \Big\{
        \bigcup_{k=1}^{m}(E_k\times J_k): &      
            \quad m\in\mathbb N,\quad E_k\in\mathfrak B([0,T]), \quad J_k=[a_k,b_k)\cap[-1,0], \\ & 
            a_k,b_k\in\mathbb R,\quad a_k<b_k,\quad E_k\times J_k\text{ are pairwise disjoint}\Big\}. 
\end{split}
\end{equation*}
   In other words we are going to show $\mathcal{R}\subset \mathcal{M}$. Therefore, let $A\in \mathcal{R}$ so $A=(E_1\times [a_1,b_1))\cup\dots\cup (E_m\times [a_m,b_m))$ for some Borel sets $E_k$ and some intervals $[a_k,b_k)$ (and the union is by definition disjoint). In order to show that $A\in \mathcal{M}$ we will show that the mapping 
   \begin{equation*}
       [\,[0,T]\ni\tau\mapsto \mu_{\tau}(A_{\tau})\in \RR\,]
   \end{equation*}
   is measurable. To do that note that, for each $\tau\in [0,T]$ we have 
   \begin{align*}
       \mu_{\tau}(A_{\tau}) & = \mu_{\tau}\Big(\Big( \bigcup_{k=1}^m (E_k\times [a_k,b_k))\Big)_{\tau}\Big) \\
       &= \sum_{k=1}^m \mu_{\tau}( (E_k\times [a_k,b_k))_{\tau} ) \\
          &= \sum_{k=1}^m \mathbbm{1}_{E_k}(\tau) \mu_{\tau}([a_k,b_k)). 
   \end{align*}
Hence, it remains to show that the mappings 
\begin{equation} \label{eq:ring-map}
    [\, [0,T]\ni \tau \mapsto  \mu_{\tau}([a_k,b_k)) \in \RR \, ]
\end{equation}
are measurable (for each $k=1,\dots,m$). Without loss of generality we assume that $\mu_{\tau}$ is a positive measure; if not we can decompose measure $\mu_{\tau}$ first and proceed separately. We start by defining the sequence of trapezoid-shaped function $(\Pi_m)_{m=1}^{\infty}\subset C([-1,0])$ by 
\begin{equation*}
    \Pi_m(s; a_k, b_k) = \begin{cases}  0 & \text{for } s \in [-1, a_k - 1/m], \\  m(s - a_k) + 1 & \text{for } s \in [a_k - 1/m, a_k], \\  1 & \text{for } s \in [a_k, b_k - 1/m], \\  m(b_k - s) & \text{for } s \in [b_k - 1/m, b_k], \\  0 & \text{for } s \in [b_k, 0].  \end{cases}
\end{equation*}
Lastly we have the following inequality 
\begin{equation*}
  \mu_{\tau}( [a_k,b_k-1/m) )  \le \langle \Pi_m(\cdot;a_k,b_k), \mu(\cdot ; \tau)  \rangle \le \mu_{\tau}([a_k-1/m,b_k) ).
\end{equation*}
Therefore, due to lower and upper semicontinuity of the measure $\mu_{\tau}$ and the squeeze theorem we have the following 
\begin{equation*}
    \lim_{m\to\infty} \langle \Pi_m(\cdot;a_k,b_k), \mu_{\tau} \rangle  = \mu_{\tau}([a_k,b_k)). 
\end{equation*}
This together with the fact that for all $m\in\NN$ the mapping 
\begin{equation*}
    [\, [0,T]\ni \tau \mapsto \langle \Pi_m(\cdot;a_k,b_k), \mu_{\tau}  \rangle \in \RR \,]
\end{equation*} 
is measurable (see assumption~\ref{A2-ODEs}) shows that the mapping~\ref{eq:ring-map} is measurable. Therefore, $A\in\mathcal{M}$. By the monotone class theorem, the $\sigma$\nobreakdash-\hspace{0pt}ring generated by the ring $\mathcal{R}$ is contained in the monotone class $\mathcal{M}$. Since $[0,T]\times [-1,0] \in \mathcal{R}$, the ring $\mathcal{R}$ is an algebra and the $\sigma$\nobreakdash-\hspace{0pt}ring is a  $\sigma$\nobreakdash-\hspace{0pt}algebra $\mathfrak{B}([0,T]\times [-1,0])$.      
\end{proof}
The above lemma shows that the mapping 
\begin{equation*}
    [\, \mathfrak{B}([0,t]\times [-1,0]) \ni A \mapsto \nu_t(A) \in \RR \,]
\end{equation*} 
is well defined for any $t\in [0,T]$. In other words, for all $A \in \mathfrak{B}([0,T]\times [-1,0])$ the mapping 
\begin{equation*}
    [\,[0,t]\ni \tau \mapsto \mu_{\tau}(A_{\tau})\in\RR\,]
\end{equation*}
is in $L_{\infty}([0,t])$. The next lemma will show that the mapping $\nu_t$ is indeed a signed measure.  
\begin{lemma} For any $t\in[0,T]$ the assignment $\nu_t:\mathfrak{B}([0,T]\times [-1,0]) \to \RR$ is a signed measure. \end{lemma}
\begin{proof}
 Fix $t$. Clearly, $\nu_t(\varnothing)=0$. Let $(A_m)_{m=1}^{\infty}\subset \mathfrak{B}([0,T]\times [-1,0])$ be a sequence of pairwise disjoint sets. We have the following 
    \begin{align*}
          \nu_t\Big(\bigcup_{m=1}^{\infty} A_m \Big) & = \int_0^t \mu_{\tau} \Big( \Big(\bigcup_{m=1}^{\infty} A_m\Big)_{\tau} \Big) \,\dd \tau  
           = \int_0^t \sum_{m=1}^{\infty} \mu_{\tau}( (A_m)_{\tau}) \,\dd \tau  \\
          & =  \sum_{m=1}^{\infty} \int_0^t \mu_{\tau}( (A_m)_{\tau}) \,\dd \tau  
           =  \sum_{m=1}^{\infty}   \nu_t( A_m).    
    \end{align*}
Note that interchanging the sum with the integral follows from Lebesgue's dominated convergence theorem. Indeed, the sequence of measurable mappings 
\begin{equation*}
    \Big(\Big[\,[0,t]\ni\tau\mapsto\sum_{m=1}^{L}\mu_{\tau} ( (A_m)_{\tau}) \,\Big]\Big)_{L=1}^{\infty}
\end{equation*}
is bounded, since for any $L$ and $\tau$ we have  
\begin{align*}
   \Big| \sum_{m=1}^{L}\mu_{\tau} ( (A_m)_{\tau}) \Big| & \le \sum_{m=1}^{L} | \mu_{\tau} | ( (A_m)_{\tau}) \\
   & = | \mu_{\tau}| \Big( \Big( \bigcup_{m=1}^{\infty} A_m\Big)_{\tau}\Big) 
      \le K.
\end{align*}
\end{proof}

\begin{lemma}\label{lemma:desint-positive-is-positive}
    Let $\beta$ be a positive measure on $\mathfrak{B}([0,T]\times[-1,0])$ such that 
    \begin{equation*}
        \beta(A)=\int_{0}^T \alpha_{\tau}(A_\tau) \,\dd \tau
    \end{equation*}
    for some mapping $[\,[0,T]\ni\tau\mapsto \alpha_{\tau}(\cdot _\tau)\in \mathfrak{m}\,]$ such that for any Borel set $A\in [0,T]\times[-1,0]$ the mapping $[\,[0,T]\ni\tau\mapsto \alpha_{\tau}(A_\tau)\in \RR\,]$ is  $(\mathfrak{L}([0, T]), \mathfrak{B}(\RR))$\nobreakdash-\hspace{0pt}measurable. Then for $\ell$-a.e. $\tau\in [0,T]$ the measures $\alpha_{\tau}$ are positive. 
\end{lemma}
\begin{proof}
    Let $\mathfrak{I}=\{ [a,b)\cap [-1,0] : a,b\in\QQ \ \& \ a<b \}$. Observe that straight from the assumptions we have that for any $I\in \mathfrak{I}$
    the mapping 
\begin{equation*}
    [\, [0,T]\ni \tau \mapsto \alpha_{\tau}(I) \in \RR\,]
\end{equation*}
is $(\mathfrak{L}([0, T]), \mathfrak{B}(\RR))$\nobreakdash-\hspace{0pt}measurable (compare with the proof of Lemma~\ref{lemma:borel-monotone-class-measurable}) and in addition for any $J\in \mathfrak{B}([0,T])$ we have 
\begin{equation*}
    \int_{J}\alpha_{\tau}(I) \,\dd \tau = \beta(J\times I)\ge 0. 
\end{equation*} 
Hence, for any $I\in \mathfrak{I}$; $\alpha_{\tau}(I)\ge 0$ for $\ell$-a.e.\ $\tau\in [0,T]$. Since, $|\mathfrak{I}|=\aleph_0$ we can deduce that there is a $\ell$-null-measure subset of $[0,T]$ such that for any $\tau$ in its complement and any $I\in\mathfrak{I}$ we have $\alpha_{\tau}(I)\ge 0$. It remains to show that $\alpha_{\tau}\ge 0$. It can be done via the monotone class theorem. It is easy to check that the family of sets $\{\,B\in\mathfrak{B} ([-1,0]) : \alpha_{\tau}(B)\ge 0 \, \}$ is a monotone class. From a previous observation (ring) $\mathfrak{I}$ is contained in it. Hence, by the monotone class theorem, the $\sigma$-ring generated by $\mathfrak{I}$ is contained in $\{B\in\mathfrak{B} ([-1,0]) : \alpha_{\tau}(B)\ge 0 \}$. Hence, such a $\sigma$-ring equals $\mathfrak{B} ([-1,0])$, and the proof is done.       
\end{proof}
In a similar sense, we have the following Lemma concerning a measurable property ${p}(\tau,s)$ holding $\nu$-a.e.\ on $[0,T]\times [-1,0]$ or somehow symmetrically  $\ell$-a.e.  on $[0,T]$ and $\mu_{\tau}$-a.e.\ on $[-1,0]$. It turns out that we have equivalence. 

\begin{lemma}\label{lemma:p-formula}
Let $p(t,s)$ be a property such that
\begin{equation*}
        \{(t,s)\in[0,T]\times[-1,0]:p(t,s)\}
    \in\mathfrak B([0,T]\times[-1,0]).
\end{equation*}
Then the following statements are equivalent:
\begin{enumerate}[label=\textup{(}\roman*\textup{)}, ref=\textup{(}\emph{\roman*}\textup{)}]
    \item\label{lemma:p-formula-i} $p(t,s)$ holds for $\nu$-almost every
          $(t,s)\in[0,T]\times[-1,0]$;
    \item\label{lemma:p-formula-ii} for $\ell$-almost every $t\in[0,T]$,
          $p(t,s)$ holds for $\mu_t$-almost every
          $s\in[-1,0]$.
\end{enumerate}
\end{lemma}

\begin{proof}
Define the Borel set
\begin{equation*}
        Z:=\{(t,s)\in [0,T]\times[-1,0]: p(t,s)
         \text{ does not hold}\}.
\end{equation*}
The two assertions are equivalent, respectively, to
\begin{equation*}
    |\nu|(Z)=0
    \qquad\text{and}\qquad
    |\mu_t|(Z_t)=0
    \quad\text{for $\ell$-almost every }t\in [0,T].
\end{equation*}
We prove the equivalence of these conditions. For each $t\in I$, define the signed measure $\alpha_t(B):=\mu_t(B\cap Z_t)$ for $B\in\mathfrak B([-1,0])$. 
Its total variation satisfies $\|\alpha_t\|_{\mathrm{TV}}
\leq\|\mu_t\|_{\mathrm{TV}}\leq K$. 
Moreover, for every $A\in\mathfrak B([0,T]\times [-1,0])$, we have $\alpha_t(A_t)=\mu_t((A\cap Z)_t)$. 
Therefore, by Lemma~\ref{lemma:borel-monotone-class-measurable}, the mapping $[ \, t \mapsto \alpha_t(A_t)\, ]$ is $(\mathcal L([0,T]),\mathfrak B(\RR))$-measurable. Thus both $(\alpha_t)_{t\in [0,T]}$
and $(-\alpha_t)_{t\in [0,T]}$ satisfy
assumptions of Lemma~\ref{lemma:desint-positive-is-positive}.

\ref{lemma:p-formula-i}$\Rightarrow$\ref{lemma:p-formula-ii}.  Suppose that $|\nu|(Z)=0$. Then
\begin{equation*}
        \beta(A):=\nu(A\cap Z)
             =\int_0^T\alpha_t(A_t)\,\dd t
\end{equation*}
is the zero measure. In particular, both $\beta$
and $-\beta$ are nonnegative.
Applying Lemma~\ref{lemma:desint-positive-is-positive} to the representations of
$\beta$ and $-\beta$, we obtain $\alpha_t\geq0 $ and $-\alpha_t\geq0$ 
outside a common $\ell$-null set.
Therefore, $\alpha_t=0$ as a measure for almost
every $t$, and hence $ |\mu_t|(Z_t)=0$
for $\ell$-almost every $t$. 

\ref{lemma:p-formula-ii}$\Rightarrow$\ref{lemma:p-formula-i}. Conversely, suppose that $|\mu_t|(Z_t)=0$
for $\ell$-almost every $t$.
For every $A\in\mathfrak B([0,T]\times [-1,0])$ we have $   |\mu_t(A_t\cap Z_t)|
    \leq|\mu_t|(Z_t)=0$ 
for almost every $t$. Consequently,
\begin{equation*}
        \nu(A\cap Z)
    =\int_{0}^T\mu_t(A_t\cap Z_t)\,\dd t
    =0.
\end{equation*}
Thus the restriction of $\nu$ to $Z$ is the zero
measure, so $|\nu|(Z)=0$. 
\end{proof}

\begin{lemma}\label{lemma:Hahn-fact}
    Let $(X,\Sigma,\alpha)$ be a measure space. Moreover, let $ \widetilde{P}, \widetilde{N}\in\Sigma$ be such that $\widetilde{P}\cap \widetilde{N}=\varnothing$, $\widetilde{P}\cup \widetilde{N}=X$ and for any $Y\in\Sigma$ we have $\alpha(Y\cap  \widetilde{P})\ge 0$ and $\alpha(Y\cap \widetilde{N})\le 0$. Then $\alpha^+=\alpha(\cdot \cap \widetilde{P})$ and $\alpha^-=-\alpha(\cdot \cap \widetilde{N})$, or, in other words, $\widetilde{P}, \widetilde{N}\in\Sigma$ are positive and negative sets of the measure $\alpha$. 
\end{lemma}
\begin{proof} Compare with~\cite[p. 86, 3.3]{Folland}.   
\end{proof}

\begin{remark}
    Therefore, from now on, until the end of this section, we will drop the subscript of the measure.   
\end{remark}

Finally we are going to show that the equality 
 \begin{equation*}
  \int_0^T \int_{[-1,0]} f(\tau,s)  \dd \mu_{\tau} (s) \,\dd\tau = \int_{[0,T]\times [-1,0]} f(\tau,s)  \,\dd \nu(\tau,s)
\end{equation*}
holds for any $f\in L_1(\nu)$. We start from two simpler statements, the first for $L_{\infty}(\nu)$ functions and the second for $L_1(\nu)$ under an additional assumption that for $\ell$-a.e.\ $t$ the measures $\mu_t$ are positive.       

\begin{lemma}\label{lemma:L-inf-equality} For any $f\in L_{\infty}(\nu)$ the equality 
    \begin{equation*}
  \int_0^T \int_{[-1,0]} f(\tau,s)  \dd \mu_{\tau} (s) \,\dd\tau = \int_{[0,T]\times [-1,0]} f(\tau,s)\nu(\dd \tau \dd s) 
\end{equation*}
holds.  
\end{lemma}
\begin{proof} For any $A\in \mathfrak{B}([0, T]\times [-1,0])$ we have 
\begin{multline*}
      \int_{[0,T]\times [-1,0]} \mathbbm{1}_{A}(\tau,s) \,\dd \nu(\tau,s)  = \nu(A) 
    \\ = \int_0^{T} \mu_{\tau}((A_{\tau}))\,\dd \tau
    =\int_0^{T}\int_{[-1,0]} \mathbbm{1}_{A}(\tau,s)\, \dd \mu_{\tau} (s) \,\dd \tau.
\end{multline*}
Hence, for any simple function, the statement holds. Let $f\in L_{\infty}(\nu)$ (we think of $f$ as a function, not an equivalence class) and $(f_m)_{m=1}^{\infty}\subset L_{\infty}(\nu)$ be a sequence of simple (uniformly bounded in $m$) functions such that $f_{m}\to f$ and $|f_m|\le |f|$ pointwise for all $(\tau,s)\in [0,T]\times [-1,0]$ (see~\cite[Thm. 2.10(b)]{Folland}). Note that $L_{\infty}(\nu)$ functions are by definition $(\mathfrak{B}([0, T]\times [-1,0]), \mathfrak{B}(\RR))$\nobreakdash-\hspace{0pt}measurable.
A triple use of Lebesgue's dominated convergence theorem will conclude the proof. Indeed, firstly, we have 
\begin{equation*}
     \int_{[0,T]\times [-1,0]} f_{m}(\tau,s)\, \,\dd \nu(\tau,s)\to \int_{[0,T]\times [-1,0]} f(\tau,s) \nu(\dd \tau \dd s ).  
\end{equation*}
On the other hand, for each $0 \le \tau \le T$, by Lebesgue's dominated convergence theorem, we have 
\begin{equation*}
  \int_{ [-1,0]} f_{m}(\tau,s)\, \dd \mu_{\tau} (s)  \to \int_{ [-1,0]} f(\tau,s)\, \dd \mu_{\tau} (s). 
\end{equation*}  
Again, by Lebesgue's dominated convergence theorem, we have
\begin{equation*} 
    \int_0^{T}\int_{ [-1,0]} f_m(\tau,s)\, \dd \mu_{\tau} (s) \,\dd \tau \to \int_0^{T}\int_{ [-1,0]} f(\tau,s)\, \dd \mu_{\tau} (s) \,\dd \tau. 
\end{equation*}
as $m\to\infty$. 
\end{proof} 

\begin{lemma}\label{lemma:L1-equality-positive-measures} For any $\mu\in\mathfrak{m}_T$ such that $\mu_t$ is positive and $f\in L_{1}(\nu)$ the equality 
    \begin{equation*}
  \int_0^T \int_{[-1,0]} f(\tau,s)  \dd \mu_{\tau} (s) \,\dd\tau = \int_{[0,T]\times [-1,0]} f(\tau,s)\nu(\dd \tau  \dd s) 
\end{equation*}
holds.  
\end{lemma}
\begin{proof}
   Fix $f\in L_1(\nu)$ and in addition assume that $f\ge 0$. Let $(f_{m})_{m=1}^{\infty}$ be a monotone non-decreasing sequence of $L_{\infty}(\nu)$ functions such that $f_{m}\to f$ point-wise on $[0,T]\times [-1,0]$. By Lemma~\ref{lemma:L-inf-equality} we know that for each $m\in \NN$ we have 
   \begin{equation*}
       \int_{[0,T]\times [-1,0]} f_{m}(\tau,s) \, \dd \nu(\tau,s)= \int_0^{T}\int_{[-1,0]} f_n(\tau,s)\, \dd \mu_{\tau} (s) \,\dd \tau. 
\end{equation*} 
By Lebesgue's dominated convergence, theorem we have 
\begin{equation*}
     \int_{[0,T]\times [-1,0]} f_{m}(\tau,s) \, \nu(\dd \tau\dd s )\to  \int_{[0,T]\times [-1,0]} f(\tau,s) \, \nu(\dd \tau\dd s ).
\end{equation*}
On the other hand, for any $m\in\NN$ and $\ell$-a.e.\ $\tau$ the integral   
\begin{equation*}
    \int_{[-1,0]} f_{m}(\tau,s)\, \dd \mu_{\tau} (s)  
\end{equation*}
exists (is finite). Otherwise it would be $\infty$ contradicting  
\begin{equation*}
     \int_{[0,T]\times [-1,0]} f_m(\tau,s)  \,\dd \nu(\tau,s)<\infty. 
\end{equation*}
Also for $\ell$-a.e.\ $\tau$, by Lebesgue's monotone convergence theorem,  
\begin{equation*}
    \int_{[-1,0]} f_m(\tau,s)\, \dd \mu_{\tau} (s)  \to  \int_{[-1,0]} f(\tau,s)\, \dd \mu_{\tau} (s). 
\end{equation*}
Furthermore, a stronger bound can be obtained (we have uniform bound in $m\in\NN$) since we have 
\begin{align*}
    \int_0^{T}\int_{[-1,0]} f_m(\tau,s)\, \dd \mu_{\tau} (s) \,\dd \tau & = \int_{[0,T]\times [-1,0]} f_m(\tau,s) \, \dd \nu(\tau,s)\\ &
    \le  \int_{[0,T]\times [-1,0]} f(\tau,s) \, \nu(\dd \tau\dd s )=:\widetilde{M}.
\end{align*}
Therefore, by the Fatou Lemma, 
\begin{equation*}
     \int_0^{T}\int_{[-1,0]} f(\tau,s)\, \dd \mu_{\tau} (s) \,\dd \tau \le \liminf_{m\to\infty} \int_0^{T}\int_{[-1,0]} f_m(\tau,s)\, \dd \mu_{\tau} (s) \,\dd \tau  \le \widetilde{M}. 
\end{equation*}
Hence, the assumptions of Lebesgue's dominated convergence theorem are fulfilled, which concludes the proof for non-negative functions, as
\begin{equation*}
     \int_0^{T}\int_{[-1,0]} f_m(\tau,s)\, \dd \mu_{\tau} (s) \,\dd \tau \to\int_0^{T}\int_{[-1,0]} f(\tau,s)\, \dd \mu_{\tau} (s) \,\dd \tau.
\end{equation*}
In general, for not necessarily non-negative $f\in L_1(\nu)$ we write 
\begin{multline*}
        \int_{[0,T]\times [-1,0]} f(\tau,s) \, \nu(\dd \tau  \dd s ) \\ = \int_{[0,T]\times [-1,0]} f^+(\tau,s) \, \nu(\dd \tau  \dd s ) - \int_{[0,T]\times [-1,0]} f^-(\tau,s) \, \nu(\dd \tau  \dd s ) \\ 
         = \int_0^{T}\int_{[-1,0]} f^+(\tau,s)\, \dd \mu_{\tau} (s) \,\dd \tau -  \int_0^{T}\int_{[-1,0]} f^-(\tau,s)\, \dd \mu_{\tau} (s) \,\dd \tau \\ 
         = \int_0^{T}\int_{[-1,0]} f(\tau,s)\, \dd \mu_{\tau} (s) \,\dd \tau. 
\end{multline*}
\end{proof}
We are going to show that the assumption concerning positivity in Lemma~\ref{lemma:L1-equality-positive-measures} can be dropped. However, we have to be careful about the existence of the inner integral for $\ell$-a.e.\ $\tau$. We need a general result concerning the non-negativity of a kernel (disintegration) of positive measure.      

\begin{lemma}\label{lemma:mu+/-_desint}~
 \begin{equation*}
  \nu^{\pm}= \int_{0}^T \mu_{\tau}^{\pm}(\cdot_{\tau})\,\dd \tau. 
\end{equation*} 
\end{lemma}
\begin{proof}
   Let $P^{\nu}$ and $N^{\nu}$ be the positive and negative sets of the measure $\nu$.  For any $A\in \mathfrak{B}([0,T]\times[-1,0])$, from the definition of the Hahn--Jordan decomposition and the measure $\nu$ we can write 
   \begin{equation*}
       \nu^+(A)= \nu(A\cap P^{\nu}) = \int_{0}^T \mu_{\tau}(A_{\tau}\cap P^{\nu}_{\tau})\,\dd \tau. 
   \end{equation*}
   From Lemma~\ref{lemma:desint-positive-is-positive} we know that the measure $\mu_{\tau}(\cdot \cap P^{\nu}_{\tau})$ is a positive measure on $\mathfrak{B} ([0,T])$ for $\ell$-a.e.\ $\tau\in [0,T]$. Similarly, $\mu_{\tau}(\cdot \cap N^{\nu}_{\tau})\le 0$ for $\ell$-a.e.\ $\tau\in [0,T]$. Therefore, by Lemma~\ref{lemma:Hahn-fact}  
   \begin{equation*}
    \mu_{\tau}(\cdot \cap P^{\nu}_{\tau}) = \mu_{\tau}^+ \quad \& \quad   -\mu_{\tau}(\cdot \cap N^{\nu}_{\tau}) = \mu^-
   \end{equation*}
   for $\ell$-a.e.\ $\tau\in [0,T]$. 
\end{proof}
\begin{lemma}
       For any positive set $P^{\nu}$ of the measure $\nu$ for $\ell$-a.e. $\tau\in [0,T]$ the $\tau$-section of $P^{\nu}$ (denoted by $P^{\nu}_{\tau}$) is a positive set of the measure $\mu_{\tau}$. The same for negative sets.      
\end{lemma}
\begin{proof}
   The statement is a translation of the results from Lemma~\ref{lemma:mu+/-_desint}. Indeed, let $P^{\nu}$ be the positive set of the measure $\nu$. From the proof of Lemma~\ref{lemma:mu+/-_desint} we know that for $\ell$-a.e. $\tau\in [0,T]$ we have $\mu_{\tau}(\cdot \cap P^{\nu}_{\tau}) = \mu^+$. Hence, straight by the definition of positive set, $P^{\nu}_{\tau}$ is the one for the measure $\mu_{\tau}$. 
\end{proof}
\begin{lemma}\label{lemma:nu-mi-l1-for_+/-}
    For any $f\in L_1(\nu)$ 
 \begin{equation*}
  \int_0^T \int_{[-1,0]} f(\tau,s) \dd \mu_{\tau}^{\pm} (s) \,\dd\tau = \int_{[0,T]\times [-1,0]} f(\tau,s) \nu^{\pm} (\dd \tau\dd s ). 
\end{equation*} 
\begin{proof}
    The statement follows from Lemma~\ref{lemma:L1-equality-positive-measures} and Lemma~\ref{lemma:mu+/-_desint}.
\end{proof}
\end{lemma}
\begin{theorem}\label{thm:integrated-eq}
    For any $f\in L_1(\nu)$ 
 \begin{equation*}
  \int_0^T \int_{[-1,0]} f(\tau,s) \dd \mu_{\tau} (s) \,\dd\tau = \int_{[0,T]\times [-1,0]} f(\tau,s) \, \dd \nu(\tau,s) 
\end{equation*} 
\end{theorem}
\begin{proof}
    Follows directly from Lemma~\ref{lemma:nu-mi-l1-for_+/-}. 
\end{proof}
\begin{lemma} \label{lemma:nu-modulus}
For any $f\in L_1(\nu)$  
\begin{equation*}
  \int_0^T \int_{[-1,0]} f(\tau,s) \dd |\mu_{\tau}| (s) \,\dd\tau = \int_{[0,T]\times [-1,0]} f(\tau,s) |\nu|(\dd \tau\dd s )  
\end{equation*}  
\end{lemma}
\begin{proof}
     Follows directly from Lemma~\ref{lemma:nu-mi-l1-for_+/-}. 
\end{proof}

\section{Integral formulation via the integrated measure}\label{sec:Integral-formulation}

The integrated measure provides an alternative way to establish
the contraction property of the continuation operator.
Although existence and uniqueness have already been proved,
the following argument shows how the same conclusion follows
directly from the measure formulation.

\begin{proposition}
\label{prop:integrated-contraction}
Let $\mu\in\mathfrak{m}_T$ and $u_0\in C([-1,0])$.
For every $\kappa>K$, the continuation operator $\mathfrak{G}_\mu:
C_{u_0}([0,T])\to C_{u_0}([0,T])$
is a contraction with respect to $d_\kappa(z,w)$. 
More precisely for any $z,w\in C_{u_0}([0,T])$ we have 
\begin{equation*}
d_\kappa(\mathfrak{G}_\mu z,\mathfrak{G}_\mu w)
\le
\frac{K}{\kappa}\,d_\kappa(z,w). 
\end{equation*}
In particular, the contraction constant is independent of
$\mu$ and $u_0$.
\end{proposition}

\begin{proof}
Fix $z,w\in C_{u_0}([0,T])$ and extend both functions to
$[-1,0]$ by the same initial history $u_0$.
For $t\in[0,T]$, the representation using the integrated
measure gives
\begin{equation*}
\mathfrak{G}_\mu (z)[t]- \mathfrak{G}_\mu (w)[t]
=
\int_{[-1,t]}(z(r)-w(r))\,
\dd((\cdot+\cdot)_*\nu_t)(r).
\end{equation*}
Here $(\cdot+\cdot)_*\nu_t$ denotes the pushforward of
$\nu_t$ under $(\tau,s)\mapsto\tau+s$.
The integrand vanishes for $r\in[-1,0]$, while for any $r\in[0,T]$ we have 
\begin{equation*}
|z(r)-w(r)|
\le e^{\kappa r}d_\kappa(z,w). 
\end{equation*}
Consequently,
\begin{equation*}
e^{-\kappa t}
|\mathfrak{G}_\mu (z)[t]- \mathfrak{G}_\mu (w)[t]|
\le
d_\kappa(z,w)
\int_{[0,t]}e^{\kappa(r-t)}
\,\dd |(\cdot+\cdot)_*\nu_t|(r).
\end{equation*}
It remains to estimate the last integral. Using $|(\cdot+\cdot)_*\nu_t|
\le
(\cdot+\cdot)_*|\nu_t|$
and the representation of $|\nu_t|$ established in Lemma~\ref{lemma:nu-modulus}, we obtain
\begin{equation*}
\begin{aligned}
&\int_{[0,t]}e^{\kappa(r-t)}
\,\dd|(\cdot+\cdot)_*\nu_t|(r)
\\
&\quad\le
\int_{[0,t]\times[-1,0]}
\mathbbm{1}_{[0,t]}(\tau+s)e^{\kappa(\tau+s-t)}
\,\dd|\nu_t|(\tau,s)
\\
&\quad=
\int_0^t\int_{[-1,0]}
\mathbbm{1}_{[0,t]}(\tau+s)e^{\kappa(\tau+s-t)}
\,\dd|\mu_\tau|(s)\,\dd\tau
\\
&\quad\le
\int_0^t e^{\kappa(\tau-t)}
\|\mu_\tau\|_{\mathrm{TV}}\,\dd\tau
\\
&\quad\le
K\int_0^t e^{\kappa(\tau-t)}\,\dd\tau
=
\frac{K}{\kappa}(1-e^{-\kappa t})
\le \frac{K}{\kappa}.
\end{aligned}
\end{equation*}
The second inequality uses $s\le0$ and $\kappa>0$. Taking the supremum over $t\in[0,T]$ proves the claimed
estimate. Since $K/\kappa<1$, the operator
$\mathfrak{G}_\mu$ is a contraction.
\end{proof}

The Banach fixed-point theorem therefore recovers the unique
fixed point of $\mathfrak{G}_\mu$, and hence the unique solution
of the original initial-value problem.

{
\section{Measurable linear skew-product semidynamical systems}\label{sec:skew-product}
We next consider coefficient paths generated by a
measure-preserving base flow.
The deterministic results yield linear solution operators,
the cocycle identity, and continuity in time and
the initial history.
Measurability with respect to the base point requires
a separate analysis of the coefficient space and
is established in Section~\ref{sec:Measurable-structures}.

\subsection{Measurable dynamical systems}\label{subsec:MDS}
We write $\RR^{+}$ for $[0, \infty)$. For a metric space $S$ by $\mathfrak{B}(S)$ we denote the $\sigma$\nobreakdash-\hspace{0pt}algebra of Borel subsets of $S$. A probability space is a triple $\OFP$, where $\Omega$ is a set, $\mathfrak{F}$ is a $\sigma$\nobreakdash-\hspace{0pt}algebra of subsets of $\Omega$, and $\PP$ is a probability measure defined for all $F \in \mathfrak{F}$.  We always assume that the measure $\PP$ is complete.
\par\smallskip
A {\em \textup{(}two\nobreakdash-\hspace{0pt}sided\textup{)}measurable dynamical system} on the probability space $\OFP$ is a $(\mathfrak{B}(\RR) \otimes \mathfrak{F},\mathfrak{F})$\nobreakdash-\hspace{0pt}measurable mapping $\theta\colon\RR\times \Omega\to \Omega$ such that
\begin{itemize}[label=\raisebox{0.25ex}{\tiny$\bullet$}]
\item $\theta(0,\omega)=\omega$ for any $\omega  \in \Omega$,
\item $\theta(t+s,w)=\theta(t,\theta(s,\omega))$ for any $\omega  \in \Omega$ and $t,\,s \in \RR$.  
\end{itemize}
We write $\theta(t,\omega)$ as $\theta_t\omega$. Also, we usually denote measurable dynamical systems by $(\OFP,(\theta_{t})_{t \in \RR})$ or simply by $(\theta_{t})_{t \in \RR}$. 
\par
A {\em metric dynamical system} is a measurable dynamical system $(\OFP,(\theta_{t})_{t \in \RR})$
such that for each $t\in\RR$  the mapping $\theta_t\colon \Omega\to\Omega$ is $\PP$-preserving (i.e., $\PP(\theta_t^{-1}(F))=\PP(F)$ for any $F\in\mathfrak{F}$ and $t\in\RR$).
A subset $\Omega'\subset\Omega$ is \emph{invariant} if $\theta_t(\Omega')=\Omega'$ for all $t\in\RR$, and the metric dynamical system is said to be \emph{ergodic} if for any invariant subset  $F \in \mathfrak{F}$, either $\PP(F) = 1$ or $\PP(F) = 0$. Throughout the paper we will assume that $\PP$ is ergodic.

\subsection{Measurable linear skew-product semidynamical systems}\label{subsec:MLSPS} 

By a {\em measurable linear skew-product semidynamical system} or {\em semiflow}, 
$\Phi = \allowbreak ((U_\omega(t))_{\omega \in \Omega, t \in \RR^{+}}, \allowbreak (\theta_t)_{t\in\RR})$ on  $X$ covering a metric dynamical system $(\theta_{t})_{t \in \RR}$ we understand a $(\mathfrak{B}(\RR^{+}) \otimes \mathfrak{F} \otimes \mathfrak{B}(X), \mathfrak{B}(X))$\nobreakdash-\hspace{0pt}measurable
mapping
\begin{equation*}
 \RR^{+} \times \Omega \times X \ni (t,\omega,u) \mapsto U_{\omega}(t)\,u \in X 
\end{equation*}
satisfying
\begin{align}
&U_{\omega}(0) = \mathrm{Id}_{X} \quad & \textrm{for each }\,\omega  \in \Omega, \nonumber 
\\
&U_{\theta_{s}\omega}(t) \circ U_{\omega}(s)= U_{\omega}(t+s) \qquad &\textrm{for each } \,\omega \in \Omega \textrm{ and }  t,\,s \in \RR^{+},
\label{eq-cocycle}
\\
&[\, X \ni u \mapsto U_{\omega}(t)u \in X \,] \in \mathcal{L}(X) & \textrm{for each }\,\omega \in \Omega \textrm{ and } t \in \RR^{+}.\nonumber
\end{align}
Equation~\eqref{eq-cocycle} is called the cocycle property.

\subsection{Lyapunov exponents}
Let $\Phi = \allowbreak ((U_\omega(t))_{\omega \in \Omega, t \in \RR^{+}}, \allowbreak (\theta_t)_{t\in\RR})$ be a measurable
linear semiflow on a separable Banach space $X$, covering an
ergodic metric dynamical system on a complete probability
space $(\Omega,\mathfrak F,\mathbb P)$.
For $\omega\in\Omega$ and $u\in X\setminus\{0\}$, the
Lyapunov exponent associated with $u$ is defined by
\begin{equation*}
    \lambda(\omega,u)
    :=
    \lim_{t\to\infty}
    \frac{1}{t}\log\|U_\omega(t)u\|_X,
\end{equation*}
whenever this limit exists in the extended real line.
We adopt the convention $\log 0=-\infty$.
Without assuming existence of the limit, the corresponding
upper Lyapunov exponent is defined by
\begin{equation*}
    \overline\lambda(\omega,u)
    :=
    \limsup_{t\to\infty}
    \frac{1}{t}\log\|U_\omega(t)u\|_X.
\end{equation*}

The top Lyapunov exponent describes the asymptotic growth
of the operator norm. Suppose that the functions
\begin{equation*}
    \omega\mapsto
    \sup_{0\leq s\leq1}
    \log^+\|U_\omega(s)\|_{\mathcal L(X)}
    \quad\text{and}\quad
    \omega\mapsto
    \sup_{0\leq s\leq1}
    \log^+\|U_{\theta_s\omega}(1-s)\|_{\mathcal L(X)}
\end{equation*}
belong to $L^1(\Omega,\mathfrak F,\mathbb P)$, where
$\log^+r:=\max\{0,\log r\}$ for $r>0$ and $\log^+0:=0$.
The subadditive ergodic theorem, together with these
short-time bounds, yields a constant
$\lambda_{\mathrm{top}}\in[-\infty,\infty)$ such that
\begin{equation*}
    \lim_{t\to\infty}
    \frac{1}{t}\log\|U_\omega(t)\|_{\mathcal L(X)}
    =
    \lambda_{\mathrm{top}}
    \qquad\text{for }\mathbb P\text{-almost every }\omega;
\end{equation*}
see \cite[Section~3]{MierczynskiNovoObaya2020}.
This constant is called the top Lyapunov exponent of the
semiflow. At every base point where this limit holds,
\begin{equation*}
    \overline\lambda(\omega,u)\leq\lambda_{\mathrm{top}},
\end{equation*}
where $u\in X\setminus\{0\}$.  For the semiflow on $C([-1,0])$, we denote the top Lyapunov
exponent by $\lambda_{\mathrm{top}}^C$.
The required integrability conditions will be verified below.
We impose the additional assumption
\begin{equation*}
    \lambda_{\mathrm{top}}^C>-\infty.
\end{equation*}
This excludes the case of superexponential decay of the
operator norm. Under this assumption, the compactness result
proved below will allow us to apply the multiplicative
ergodic theorem and obtain an Oseledets decomposition.

\subsection{Construction of the solution semiflow}
To be more specific, we will define the mapping 
\begin{equation*}
 \RR^{+} \times \Omega \times AC([-1,0]) \ni (t,\omega,u) \mapsto U_{\omega}(t)\,u \in AC([-1,0]) 
\end{equation*} 
in the following way 
\begin{equation*}
\begin{split}
U_\omega(t) u_0 ={}& \text{solution to the equation } 
\begin{cases} 
z'(t) = \displaystyle \int_{[-1,0]} z(t+r) \, \mathrm{d}\mu_{\theta_{t}\omega}(r) \\ 
z = u_0 \quad \text{on } [-1,0] 
\end{cases} \\
&\text{shifted backward by } t \text{ (i.e., } s \mapsto z(t+s) \text{) and restricted to } [-1,0].
\end{split}
\end{equation*}
We show below that this family defines a linear skew-product semidynamical system. Firstly, we can see that for any fixed $\omega\in \Omega$ and $u_0\in AC([-1,0])$, straight from the definition we have $U_{\omega}(0)u_0 = u_0$.  
\begin{lemma}\label{lemma:cocycle-U-AC}
    For each $\omega\in \Omega$ and $t,s\in \RR^+$ we have $U_{\theta_{s}\omega}(t) \circ U_{\omega}(s)= U_{\omega}(t+s)$.
\end{lemma}
\begin{proof}
Fix $\omega\in \Omega$, $t,s\in \RR^+$ and $u_0\in AC([-1,0])$. Let $x$ be the solution of 
\begin{equation*}
\begin{cases} 
z'(t) = \displaystyle \int_{[-1,0]} z(t+r) \, \mathrm{d}\mu_{\theta_{t}\omega}(r) \\ 
z = u_0 \quad \text{on } [-1,0] 
\end{cases}
\end{equation*}
and $x_s$ be $z(\cdot+s)$. Therefore, $z_s{\restriction}_{[-1,0]} = U_{\omega}(s)u_0$. Similarly, we can define $z_{t+s}$ and simplify the problem to showing $z_{t+s}{\restriction}_{[-1,0]} = U_{\theta_s \omega }(t)(z_s{\restriction}_{[-1,0]})$. However, from the definition of $U_{\theta_s \omega }(t)$ it suffices to show 
\begin{equation*}
\begin{cases}
 z_s'(t) = \displaystyle \int_{[-1,0]} z_s(t + r) \, \dd\mu_{\theta_{t+s}\omega}(r) \\
z_s = z_s{\restriction}_{[-1,0]} \quad \text{on } [-1,0]. 
\end{cases}
\end{equation*}
omitting translation by $t$ and restriction to $[-1,0]$. Since the second condition is trivially satisfied, we focus on the first one. Recall that 
\begin{equation*}
z(t) = z(0) + \int_{[0,t)} \int_{[-1,0]} z(\tau+r) \, \dd\mu_{\theta_\tau\omega}(r) \, \dd \tau.
\end{equation*}
Therefore, after changing the argument, we have 
\begin{equation*}
\begin{split}
    z(t+s) & = z(0) + \int_{[0,t+s)} \int_{[-1,0]} z(\tau+r) \, \dd\mu_{\theta_\tau\omega}(r) \, \dd \tau \\
& = z(s) + \int_{[s,t+s)} \int_{[-1,0]} z(\tau+r) \, \dd\mu_{\theta_\tau\omega}(r) \, \dd \tau.
\end{split}
\end{equation*}
After going back to $x_s$ notation we can write  
\begin{equation*}
z_s(t) = z_s(0) + \int_{[0,t)} \int_{[-1,0]} z_s(\tau+r) \, \dd\mu_{\theta_{\tau+s}\omega}(r) \, \dd \tau
\end{equation*}
or in the differential form 
\begin{equation*}
z_s'(t) = \int_{[-1,0]} z_s(t+r) \, \dd\mu_{\theta_{t+s}\omega}(r). 
\end{equation*}
\end{proof}

\begin{lemma}\label{lemma:U-cont-Ac}
For each $\omega\in \Omega$ and $t\in \RR^+$ we have $U_{\omega}(t)\in \mathcal{L}(AC([-1,0]))$. 
\end{lemma}
\begin{proof} Fix $\omega$, $t$ and $u_0\in AC([-1,0])$. From the definition of $U_\omega(t)$ and Proposition~\ref{prop:Gronwal-determine-bound}
\begin{equation*}
    \begin{split}
\|U_\omega(t)u_0\|_{AC([-1,0])} & = \| z(t + \cdot){\restriction}_{[-1, 0]} \|_{AC([-1,0])} \\
&\le |z(-1)| + \int_{-1}^T |z'(\tau)| \dd\tau \\
&\le |z(-1)| + \|z'_0\|_{L_1([-1,0])} + TK e^{TK}\|u_0\|_{AC([-1,0])} \\
&= (1 + TKe^{TK})\|u_0\|_{AC([-1,0])}.
    \end{split}
\end{equation*}
\end{proof}
\begin{lemma}\label{lemma:joint-continu}
Let $\omega \in \Omega$. The function
\begin{equation*}
    \mathbb{R}^+ \times AC([-1,0]) \ni (t, z) \mapsto U_\omega(t)z \in AC([-1,0]) 
\end{equation*}
is continuous.
\end{lemma}
\begin{proof}  We first establish continuity in each variable and then prove joint continuity. Continuity of the mapping 
    \begin{equation*}
     AC([-1,0]) \ni z \mapsto U_\omega(t)z \in AC([-1,0]) 
\end{equation*}
for each fixed $t\geq0$ follows directly from Lemma~\ref{lemma:U-cont-Ac}. Fix $u_0 \in AC$, $\omega \in \Omega$, let $z: [-1, T] \to \mathbb{R}$ be the solution
\begin{equation}
z(t) = z(-1)+ \int_{-1}^t f(\tau) \,\dd\tau
\end{equation}
for some $f \in L_1([-1, T])$. Since, $  U_\omega(t)u_0 = z_t{\restriction}_{[-1, 0]} \in AC([-1,0])$ we have 
\begin{equation*}
    \begin{split}
        \|U_\omega(t+h)u_0 - U_\omega(t)u_0\|_{AC} &= \|z_{t+h}{\restriction}_{[-1, 0]} - z_t{\restriction}_{[-1, 0]}\|_{AC} \\
&= |z(t+h-1) - z(t-1)| + \| S_{t+h} f - S_t f \|_{L_1([-1, 0])}
    \end{split}
\end{equation*}
where $S$ is the translation operator given by 
$S_t f(r) = f(r+t)$. The first term converges to zero as $h \to 0$ because the solution $z(\cdot)$ is absolutely continuous. The second term converges to zero since translation is continuous in the $\mathcal{L}_{s}(L_p)$($1\le p < \infty$); see~\cite[Prop. 8.5]{Folland}

In order to prove joint continuity take the sequence $(t_n, z_n)$ converging to the pair $(t, z)$ in the space $[0, T] \times AC$. We have 
\begin{equation*}
\|U_\omega(t_n)z_n - U_\omega(t)z\|_{AC}
\le \|U_\omega(t_n)z_n - U_\omega(t_n)z\|_{AC} + \|U_\omega(t_n)z - U_\omega(t)z\|_{AC}
\end{equation*}
Let us consider both components of the sum separately.
The first component can be estimated by Lemma~\ref{lemma:U-cont-Ac} as follows 
\begin{equation*}
    \|U_\omega(t_n)z_n - U_\omega(t_n)z\|_{AC} \le (1+TKe^{KT}) \|z_n - z\|_{AC}. 
\end{equation*}
The second component converges to zero directly from the previously proven continuity of the operator with respect to time $t$. Since both components of the estimate converge to zero, the entire sum also converges to zero. 
\end{proof}

To complete the construction of the measurable linear
skew-product semiflow on $AC([-1,0])$, it remains to establish
measurability with respect to the base point. More precisely,
for every fixed $t\ge0$ and $u_0\in AC([-1,0])$, we need to
prove that the mapping
\begin{equation*}
    \Omega\ni\omega
    \mapsto U_\omega(t)u_0\in AC([-1,0])
\end{equation*}
is $(\mathfrak{F} ,\mathfrak{B} (AC([-1,0])))$-measurable.
Together with continuity with respect to time and the initial
history, this will yield the joint measurability required
in the definition of the semiflow.

The proof requires a suitable measurable structure on the
space of measure-valued parameters. Since its construction
and the associated measurability arguments involve several
auxiliary results, we develop them separately in the next
section. We introduce a $\sigma$-algebra $\Sigma_T$
on $\mathfrak m_T$ and assume that
\begin{equation*}
    \Omega\ni\omega
    \mapsto
    (\mu_{\theta_\tau\omega})_{\tau\in[0,T]}
    \in\mathfrak m_T
\end{equation*}
is $(\mathfrak{F},\Sigma_T)$-measurable.
Here $\Sigma_T$ is the smallest $\sigma$-algebra making
all scalar mappings
\begin{equation*}
    \mathfrak m_T\ni\mu
    \mapsto
    \int_0^t\int_{[-1,0]}
    z(\tau+s)\,\dd\mu_\tau(s)\,\dd\tau
    \in\mathbb R,
    \qquad
    z\in AC([-1,T]),\quad t\in[0,T],
\end{equation*}
measurable. In the next section, we give equivalent characterizations
of $\Sigma_T$ in terms of the associated measures on
$[0,T]\times[-1,0]$. We then show that this measurable
structure ensures measurability of the integral operator
for each fixed argument. Finally, using successive
approximations, we establish measurability of its fixed
point, and hence of the solution with respect to the
measure-valued parameter. Composing the resulting solution
map with the parametrization by $\omega$ will provide
the remaining measurability property.

\section{Measurable structures on the parameter space}\label{sec:Measurable-structures}
Our purpose is to equip the parameter space $\mathfrak m_T$
with a $\sigma$-algebra that ensures measurable dependence
of the integral operator on the parameter. We then use this
property to establish measurability of the solution and,
subsequently, of the associated skew-product semiflow.

Fix an initial history $u_0\in AC([-1,0])$.
We retain the spaces $C_{u_0}([0,T])$ and
$AC_{u_0}([0,T])$ and the continuation operator
$\mathfrak G_\mu$ introduced earlier.
With a slight abuse of notation, every
$x\in C_{u_0}([0,T])$ is understood to be extended to
$[-1,T]$ by setting
\begin{equation*}
    x(r):=u_0(r),
    \qquad r\in[-1,0].
\end{equation*}
In particular, whenever $\tau+s<0$, the expression
$x(\tau+s)$ denotes $u_0(\tau+s)$.
The extension agrees with the continuation at zero
because $x(0)=u_0(0)$.

With this convention, the operator is written as
\begin{equation*}
    \mathfrak G_\mu (x)[t]
    =
    u_0(0)
    +
    \int_0^t\int_{[-1,0]}
        x(\tau+s)\,\dd\mu_\tau(s)\,\dd\tau,
    \qquad t\in[0,T].
\end{equation*}
Whenever values at negative times are needed, we similarly
set
\begin{equation*}
     \mathfrak G_\mu (x)[t]:=u_0(t),
    \qquad t\in[-1,0].
\end{equation*}
These extensions are only a notational convention;
the domain and range of the continuation operator remain
spaces of functions on $[0,T]$.

We seek a $\sigma$-algebra $\Sigma_T$ on $\mathfrak m_T$
such that, for every fixed $x\in AC_{u_0}([0,T])$,
the mapping
\begin{equation*}
    \mathfrak m_T\ni\mu
    \mapsto\mathfrak G_\mu (x)
    \in AC_{u_0}([0,T])
\end{equation*}
is measurable with respect to
$\Sigma_T$ and $\mathfrak B(AC([0,T]))$.
The first lemma shows that this measurability can be
verified through scalar evaluations at fixed times,
even though the range carries the topology induced
by the $AC$ norm.

\begin{lemma}\label{lem:AC-evaluations}
Let $D\subset[a,b]$ be a countable dense set containing
$a$ and $b$. For $t\in[a,b]$, let
\begin{equation*}
    \pi_t:AC([a,b])\to\mathbb R,
    \qquad
    \pi_t(x)=x(t).
\end{equation*}
Then
\begin{equation*}
    \mathfrak{B}(AC([a,b]))
    =
    \sigma(\pi_t:t\in D).
\end{equation*}
Consequently, a mapping from an arbitrary measurable space
into $AC([a,b])$ is Borel measurable if and only if
all its evaluations at times in $D$ are measurable.
\end{lemma}

\begin{proof}
Every evaluation $\pi_t$ is continuous, since
\begin{equation*}
    |x(t)|
    \le
    |x(a)|+\int_a^b|x'(r)|\,\dd r
    =
    \|x\|_{AC([a,b])}.
\end{equation*}
It follows that
\begin{equation*}
    \sigma(\pi_t:t\in D)
    \subset
    \mathfrak{B}(AC([a,b])).
\end{equation*}
For the reverse inclusion, recall that for an absolutely
continuous function its total variation equals the integral
of the absolute value of its derivative. Moreover, by
continuity, the total variation can be computed using
partitions with points in $D$. Hence
\begin{equation*}
    \|x\|_{AC([a,b])}
    =
    |x(a)|
    +
    \sup_{\substack{
        n\ge1 \\ a=t_0<\cdots<t_n=b\\
        t_0,\ldots,t_n\in D
    }} \, \, 
    \sum_{j=1}^n|x(t_j)-x(t_{j-1})|.
\end{equation*}
The family of partitions appearing in this supremum is
countable. Fix $y\in AC([a,b])$ and $r\ge0$. The closed ball with centre $y$ and radius $r$ can therefore
be written as 
\begin{equation*}
    \begin{aligned}
        &\{x:\|x-y\|_{AC([a,b])}\le r\}\\
        &\quad=
        \bigcap_{\substack{
            n\ge1\\  a=t_0<\cdots<t_n=b\\
            t_0,\ldots,t_n\in D
        }}
        \Big\{
            x:
            |x(a)-y(a)|
            +
            \sum_{j=1}^n
            |
                (x-y)(t_j)-(x-y)(t_{j-1})
            |
            \le r
        \Big\}.
    \end{aligned}
\end{equation*}
Each set in the intersection is measurable with respect
to $\sigma(\pi_t:t\in D)$, since its definition involves
only finitely many evaluations and a continuous function
of their values. Thus every closed ball belongs to this
$\sigma$-algebra. The same is true of every open ball,
because an open ball is a countable union of concentric
closed balls with smaller radii. 
In particular, $AC([a,b])$ is separable since it is isometrically isomorphic with $\RR\times L_1([a,b])$.
Its topology has a countable base of open balls, so every
open set is a countable union of measurable balls.
This proves the reverse inclusion.
\end{proof}

\subsection{The weak-* Borel structure on the space of measures}
Set
\begin{equation*}
    Q=[0,T]\times[-1,0],
    \qquad
    \mathcal M(Q)=C(Q)^*.
\end{equation*}
We identify $\mathcal M(Q)$ with the space of finite signed
Borel measures on $Q$. It is equipped with the weak-$*$
topology $w^*$ induced by the duality with $C(Q)$. The next lemma relates Borel measurability for this topology
to scalar measurability under continuous test functions.

\begin{lemma}\label{lem:measure-Borel}
For $z\in C(Q)$ and an open set $U\subset\mathbb R$,
define
\begin{equation*}
    B_{z,U}
    =
    \{\nu\in\mathcal M(Q):\langle z,\nu\rangle\in U\}.
\end{equation*}
Then
\begin{equation*}
    \mathfrak B(\mathcal M(Q),w^*)
    =
    \sigma\bigl(
        B_{z,U}:
        z\in C(Q),\
        U\subset\mathbb R\text{ open}
    \bigr).
\end{equation*}
Consequently, a mapping into $\mathcal M(Q)$ is
weak-$*$ Borel measurable if and only if all its pairings
with functions in $C(Q)$ are measurable.
\end{lemma}

\begin{proof}
Write $\mathcal S
    =
    \sigma(
        B_{z,U}:
        z\in C(Q),\
        U\subset\mathbb R\text{ open}
    )$.  Every $B_{z,U}$ is weak-$*$ open, and hence
\begin{equation*}
    \mathcal S
    \subset
    \mathfrak B(\mathcal M(Q),w^*).
\end{equation*}
To prove the reverse inclusion, let
\begin{equation*}
    K_n
    =
    \{\nu\in\mathcal M(Q):\|\nu\|_{\mathrm{TV}}\le n\},
    \qquad n\in\mathbb N.
\end{equation*}
Choose a countable norm-dense subset $\{z_j:j\in\mathbb N\}$
of the closed unit ball of $C(Q)$.
The dual norm formula gives
\begin{equation*}
    \|\nu\|_{\mathrm{TV}}
    =
    \sup_{j\in\mathbb N}|\langle z_j,\nu\rangle|.
\end{equation*}
Therefore,
\begin{equation*}
    K_n
    =
    \bigcap_{j=1}^{\infty}
    \{\nu:|\langle z_j,\nu\rangle|\le n\}
    \in\mathcal S.
\end{equation*}

By the Banach--Alaoglu theorem and the separability of $C(Q)$,
each $K_n$ is compact and metrizable in the relative weak-$*$
topology. In particular, $K_n$ is second countable. The family
\begin{equation*}
\{K_n\cap B:B\in \{B_{z,U}:
        z\in C(Q),\
        U\subset\mathbb R\text{ open}\}\}
\end{equation*}
is a subbase for this relative topology. Thus every relatively
open set $O\subset K_n$ is a union of finite intersections
of members of this family. Since $O$ is second countable,
it is Lindel\"of, and a countable subfamily of these intersections
already covers $O$. Consequently, we may write
\begin{equation*}
O=
\bigcup_{k=1}^{\infty}
\Big(
K_n\cap\bigcap_{l=1}^{m_k}B_{k,l}
\Big),
\qquad B_{k,l}\in \{B_{z,U}:
        z\in C(Q),\
        U\subset\mathbb R\text{ open}\}.
\end{equation*}
Since $K_n\in\mathcal{S}$ and
$\{B_{z,U}:
        z\in C(Q),\
        U\subset\mathbb R\text{ open}\}\subset\mathcal{S}$, it follows that
$O\in\mathcal{S}$. Moreover, the family
\begin{equation*}
\mathcal{S}_n
:=
\{A\subset K_n:A\in\mathcal{S}\}
\end{equation*}
is a $\sigma$-algebra on $K_n$ containing all relatively open
subsets of $K_n$. Hence every Borel subset of $K_n$,
regarded as a subset of $\mathcal{M}(Q)$, belongs to
$\mathcal{S}$.

Now let $E\in\mathfrak{B}(\mathcal{M}(Q),w^*)$.
For each $n$, the set $E\cap K_n$ is Borel in $K_n$,
and therefore belongs to $\mathcal{S}$. Since
\begin{equation*}
    \mathcal{M}(Q)=\bigcup_{n=1}^{\infty}K_n, 
\end{equation*}
we obtain
\begin{equation*}
E=\bigcup_{n=1}^{\infty}(E\cap K_n)\in\mathcal{S}.
\end{equation*}
This proves the reverse inclusion and completes the proof.
\end{proof}

\subsection{Approximation by functions adapted to the delay structure}
The integral operator involves test functions of the form
$(\tau,s)\mapsto x(\tau+s)$, whereas weak-$*$ measurability
of a measure on $Q$ is expressed using arbitrary functions
in $C(Q)$.

The next lemma connects these two classes of tests.
Although functions depending only on $\tau+s$ cannot
uniformly approximate every continuous function on $Q$,
such an approximation becomes possible after allowing
a finite partition in the time variable.

\begin{lemma}\label{lem:delay-test-approximation}
For every $z\in C(Q)$, there exists a sequence
of functions
\begin{equation*}
    z_n(\tau,s)
    =
    \sum_{k=0}^{n-1}
    \mathbbm 1_{I_{k,n}}(\tau)\,x_{k,n}(\tau+s),
    \qquad
    x_{k,n}\in AC([-1,T]),
\end{equation*}
converging uniformly to $z$ on $Q$, where
\begin{equation*}
    I_{k,n}
    =
    \begin{cases}
        [kT/n,(k+1)T/n),
        & 0\le k<n-1,\\
        [(n-1)T/n,T],
        & k=n-1.
    \end{cases}
\end{equation*}
\end{lemma}

\begin{proof}
Consider the algebra
\begin{equation*}
    \mathcal A
    :=
    \Big\{
        \sum_{j=1}^{m}a_j(\tau)p_j(\tau+s):
        m\in\mathbb N,\quad a_j,p_j\in\mathbb R[r]
    \Big\}
    \subset C(Q).
\end{equation*}
It contains the constants and separates points of $Q$,
since the coordinate functions $\tau$ and $\tau+s$
determine $(\tau,s)$ uniquely.
By the Stone--Weierstrass theorem,
$\mathcal A$ is uniformly dense in $C(Q)$ (in fact $\mathcal{A}=\RR[\tau,s]$). Fix $z\in C(Q)$ and $\varepsilon>0$. Choose
\begin{equation*}
    v(\tau,s)=\sum_{j=1}^{m}a_j(\tau)p_j(\tau+s)
    \quad\text{such that}\quad
    \|z-v\|_{C(Q)}<\varepsilon/2.
\end{equation*}
Uniform continuity of the finitely many functions $a_j$
allows us to choose a finite partition $(I_k)_{k=1}^{N}$
of $[0,T]$ into intervals and constants $c_{jk}$ such that
\begin{equation*}
    \Big\|
        v-
        \sum_{k=1}^{N}\mathbbm{1}_{I_k}(\tau)
        \sum_{j=1}^{m}c_{jk}p_j(\tau+s)
    \Big\|_{C(Q)}
    <\varepsilon/2.
\end{equation*}
Setting
\begin{equation*}
    x_k(r):=\sum_{j=1}^{m}c_{jk}p_j(r),
\end{equation*}
we obtain the required approximation, with
$x_k\in AC([-1,T])$.
\end{proof}

\begin{theorem}\label{thm:equivalent-parameter-measurability}
Let $\Sigma$ be a $\sigma$-algebra on $\mathfrak m_T$.
The following conditions are equivalent:
\begin{enumerate}[label=\textup{(}\roman*\textup{)}, ref=\textup{(}\emph{\roman*}\textup{)}]
\item\label{thm:equivalent-parameter-measurability-i}    For every $x\in AC([-1,T])$ and every
    $t\in[0,T]$, the mapping
    \begin{equation*}
        \mathfrak m_T\ni\mu
        \mapsto
        \int_0^t\int_{[-1,0]}
        x(\tau+s)\,\dd\mu_\tau(s)\,\dd\tau
        \in\mathbb R
    \end{equation*}
    is $(\Sigma,\mathfrak B(\mathbb R))$-measurable.
\item\label{thm:equivalent-parameter-measurability-ii}The mapping
    \begin{equation*}
        \mathfrak m_T\ni\mu
        \mapsto\nu(\mu)\in\mathcal M(Q)
    \end{equation*}
    is $(\Sigma,\mathfrak B(\mathcal M(Q),w^*))$-measurable.
\item\label{thm:equivalent-parameter-measurability-iii}The mapping $\mu\mapsto\nu(\mu)$ is weak-$*$ measurable:
    for every $z\in C(Q)$, the scalar mapping
    \begin{equation*}
        \mathfrak m_T\ni\mu
        \mapsto\langle z,\nu(\mu)\rangle\in\mathbb R
    \end{equation*}
    is $(\Sigma,\mathfrak B(\mathbb R))$-measurable.
\end{enumerate}
\end{theorem}

\begin{proof}
The equivalence of \ref{thm:equivalent-parameter-measurability-ii} and \ref{thm:equivalent-parameter-measurability-iii} follows directly
from Lemma~\ref{lem:measure-Borel}.
It remains to prove the equivalence of \ref{thm:equivalent-parameter-measurability-i} and \ref{thm:equivalent-parameter-measurability-iii}.

\medskip
\noindent
\ref{thm:equivalent-parameter-measurability-iii}$\Rightarrow$\ref{thm:equivalent-parameter-measurability-i}.
Fix $x\in AC([-1,T])$ and $t\in[0,T]$.
Choose continuous functions $\chi_n:[0,T]\to[0,1]$
by setting
\begin{equation*}
    \chi_n(\tau)
    =
    \begin{cases}
        1,&\tau\le t,\\
        \max\{1-n(\tau-t),0\},&\tau>t.
    \end{cases}
\end{equation*}
Then $(\tau,s)\mapsto\chi_n(\tau)x(\tau+s)$ belongs to $C(Q)$. Condition \ref{thm:equivalent-parameter-measurability-iii} implies that
\begin{equation*}
    \mathfrak{m}_T\ni \mu\mapsto
    \int_0^T\int_{[-1,0]}
    \chi_n(\tau)x(\tau+s)\,
    \dd\mu_\tau(s)\,\dd\tau \in \RR
\end{equation*}
is $(\Sigma,\mathfrak B(\mathbb R))$-measurable. For each fixed $\mu$, the inner integrals multiplied
by $\chi_n(\tau)$ converge to
\begin{equation*}
    \mathbbm 1_{[0,t]}(\tau)
    \int_{[-1,0]}x(\tau+s)\,\dd\mu_\tau(s).
\end{equation*}
Their absolute values are bounded by the integrable
function
\begin{equation*}
    \tau\mapsto \|x\|_{C([-1,T])}\|\mu_\tau\|_{\mathrm{TV}}. 
\end{equation*}
The dominated convergence theorem therefore gives
\begin{equation*}
    \int_0^T\int_{[-1,0]}
    \chi_n(\tau)x(\tau+s)\,
    \dd\mu_\tau(s)\,\dd\tau
    \to
    \int_0^t\int_{[-1,0]}
    x(\tau+s)\,\dd\mu_\tau(s)\,\dd\tau.
\end{equation*}
The limit is $(\Sigma,\mathfrak B(\mathbb R))$-measurable, proving \ref{thm:equivalent-parameter-measurability-i}.

\medskip
\noindent
\ref{thm:equivalent-parameter-measurability-i}$\Rightarrow$\ref{thm:equivalent-parameter-measurability-iii}.
Fix $z\in C(Q)$.
By Lemma~\ref{lem:delay-test-approximation}, there exist
functions
\begin{equation*}
    z_n(\tau,s)
    =
    \sum_{k=0}^{n-1}
    \mathbbm 1_{I_{k,n}}(\tau)x_{k,n}(\tau+s),
    \qquad x_{k,n}\in AC([-1,T]),
\end{equation*}
converging uniformly to $z$. For each $n$, consider
\begin{equation*}
     \mathfrak{m}_T\ni \mu\mapsto
    \int_0^T\int_{[-1,0]}
    z_n(\tau,s)\,\dd\mu_\tau(s)\,\dd\tau \in \RR.
\end{equation*}
This is a finite sum of mappings of the form
\begin{equation*}
     \mathfrak{m}_T\ni \mu\mapsto
    \int_a^b\int_{[-1,0]}
    x_{k,n}(\tau+s)\,\dd\mu_\tau(s)\,\dd\tau\in \RR.
\end{equation*}
Each summand is $(\Sigma,\mathfrak B(\RR))$-measurable by \ref{thm:equivalent-parameter-measurability-i}, since it
is the difference of two integrals with upper limits
$b$ and $a$. The choice of open or closed endpoints
does not affect these integrals, whose outer integration
is with respect to Lebesgue measure.

For every fixed $\mu\in\mathfrak m_T$, we have
\begin{equation*}
    \Big|
            \int_0^T\int_{[-1,0]}
            z_n(\tau,s)\,\dd\mu_\tau(s)\,\dd\tau
            -
            \langle z,\nu(\mu)\rangle
        \Big|
        \le
        \|z_n-z\|_{C(Q)}
        \int_0^T\|\mu_\tau\|_{\mathrm{TV}}\,\dd\tau
        \to0.
\end{equation*}
Thus $\mu\mapsto\langle z,\nu(\mu)\rangle$ is a pointwise
limit of $\Sigma$-measurable functions.
This proves \ref{thm:equivalent-parameter-measurability-iii} and completes the proof.
\end{proof}

\subsection{The induced measurable structure and the integral operator}

Theorem~\ref{thm:equivalent-parameter-measurability}
allows us to introduce the smallest $\sigma$-algebra
needed for the preceding scalar integrals.
Define
\begin{equation*}
    \Sigma_T
    =
    \sigma\!\Big(
        \mu\mapsto
        \int_0^t\int_{[-1,0]}
        x(\tau+s)\,\dd\mu_\tau(s)\,\dd\tau:
        x\in AC([-1,T]),\
        t\in[0,T]
    \Big).
\end{equation*}
Equivalently, writing $\mathcal N_T(\mu)=\nu(\mu)$, we have
\begin{equation*}
    \begin{aligned}
        \Sigma_T
        &=
        \sigma\bigl(
            \mu\mapsto\langle z,\nu(\mu)\rangle:
            z\in C(Q)
        \bigr)\\
        &=
        \{
            \mathcal N_T^{-1}(B):
            B\in\mathfrak B(\mathcal M(Q),w^*)
        \}.
    \end{aligned}
\end{equation*}

The first characterization is adapted to the integral
equation, while the second and third describe the same
measurable structure in terms of measures on the
time--delay domain.

\begin{lemma}
    \label{cor:integral-operator-measurable}
Fix $u_0\in AC([-1,0])$.
For every fixed $x\in AC_{u_0}([0,T])$, the mapping
\begin{equation*}
    \mathfrak m_T\ni\mu
    \mapsto \mathfrak G_{\mu} (x)\in AC([0,T])
\end{equation*}
is $(\Sigma_T,\mathfrak B(AC([0,T])))$-measurable.
\end{lemma}

\begin{proof}
For $t\in[0,T]$, measurability of
$\mu\mapsto \mathfrak G_{\mu} (x)[t]$ follows from the definition
of $\Sigma_T$.
Lemma~\ref{lem:AC-evaluations} now gives measurability
as a mapping into $AC([0,T])$.
\end{proof}

The remaining step is to transfer this measurability
from the integral operator to its fixed point.
This step uses the continuity of $x\mapsto \mathfrak G_{\mu} (x)$
and convergence of the successive approximations,
rather than any additional measurability assumption
on the solution. Indeed, once these properties have been established,
the iterates starting from a fixed element of $AC([0,T])$
are measurable functions of $\mu$, and their limit
is the measurable solution map.

For a random parametrization, the additional assumption
is then that, for each $T>0$,
\begin{equation*}
    \Omega\ni\omega
    \mapsto
    (\mu_{\theta_t\omega})_{t\in[0,T]}
    \in\mathfrak m_T
\end{equation*}
is $(\mathfrak F,\Sigma_T)$-measurable.
Measurability of the random solution subsequently follows
by composition with the deterministic solution map.

\begin{theorem}\label{thrm:mesur-fix}
Let $(X,d)$ be a nonempty separable complete metric space,
and let $(\Omega,\mathfrak F)$ be a measurable space.
Suppose that $F:X\times\Omega\to X$ satisfies the following
conditions:
\begin{itemize}
    \item For every $x\in X$, the map
    $\omega\mapsto F(x,\omega)$ is
    $(\mathfrak F,\mathfrak B(X))$-measurable.
    \item There exists $\varkappa\in[0,1)$ such that
    \begin{equation*}
        d(F(x,\omega),F(y,\omega))
        \leq \varkappa\,d(x,y),
        \qquad x,y\in X,\quad \omega\in\Omega.
    \end{equation*}
\end{itemize}
Then each $F(\cdot,\omega)$ has a unique fixed point
$x^{\omega}$, and the map
\begin{equation*}
    \Omega\ni\omega\mapsto x^{\omega}\in X
\end{equation*}
is $(\mathfrak F,\mathfrak B(X))$-measurable.
\end{theorem}

\begin{proof}
For every $\omega\in\Omega$, the contraction assumption
implies that $F(\cdot,\omega)$ is continuous.
For every $x\in X$, the mapping $F(x,\cdot)$ is measurable
by assumption. Since $X$ is separable and metrizable,
the Carath\'eodory measurability theorem
\cite[Lemma~4.51]{AliB} implies that $ F:X\times\Omega\to X$
is $(\mathfrak B(X)\otimes\mathfrak F,\mathfrak B(X))$-measurable. Choose $x_*\in X$ and define the successive approximations
\begin{equation*}
    v_0(\omega):=x_*,
    \qquad
    v_{n+1}(\omega):=F(v_n(\omega),\omega),
    \qquad n\geq0.
\end{equation*}
We show inductively that every $v_n:\Omega\to X$ is
measurable. This is immediate for the constant map $v_0$.
If $v_n$ is measurable, then so is
\begin{equation*}
    \Omega\ni\omega\mapsto(v_n(\omega),\omega)
    \in X\times\Omega
\end{equation*}
with respect to $\mathfrak F$ and
$\mathfrak B(X)\otimes\mathfrak F$.
Composing this map with $F$ proves measurability of
$v_{n+1}$. By the Banach
fixed-point theorem $F(\cdot,\omega)$ has a unique
fixed point $x^\omega$, and
\begin{equation*}
    v_n(\omega)\to x^\omega
\end{equation*}
as $n\to\infty$. The mapping $\omega\mapsto x^\omega$ is therefore measurable
as a pointwise limit of measurable mappings into a metric
space.
\end{proof}

The same conclusion holds if $F(\cdot,\omega)$ is a
contraction with respect to a complete metric $d_\omega$
inducing the same topology as $d$.
The contraction constant may also depend on $\omega$,
provided that it is strictly less than one for every
$\omega$. Indeed, $F(\cdot,\omega)$ remains continuous with respect to $d$, so the measurable-composition argument above
still applies. The Picard iterates converge in $d_\omega$
and hence in $d$.

Applying this result to the continuation operator on the
space of absolutely continuous extensions of a fixed initial history yields measurability of the solution
with respect to $\omega$.
For each fixed time, restriction and translation of the
solution then give measurability of
\begin{equation*}
    \Omega\ni\omega
    \mapsto U_\omega(t)u_0\in AC([-1,0]).
\end{equation*}
Together with the previously established continuity
in time and in the initial history, this yields the
joint measurability required for the skew-product
semiflow.
}

\begin{proposition}\label{prop:U-measure} The mapping 
\begin{equation*}
     \Omega \ni \omega \mapsto U_\omega(t)u_0 \in AC([-1,0])
\end{equation*}
is $(\mathfrak{F},\mathfrak{B}(AC([-1,0])))$-measurable.
\end{proposition}

\begin{proposition}\label{prop:U-joint-measure} The mapping 
    \begin{equation*}
 \RR^{+} \times \Omega \times AC([-1,0]) \ni (t,\omega,u) \mapsto U_{\omega}(t)\,u \in AC([-1,0]) 
\end{equation*} 
is $(\mathfrak{B}(\RR)\otimes\mathfrak{F}\otimes \mathfrak{B}(AC([-1,0])),\mathfrak{B}(AC([-1,0])))$-measurable.
\end{proposition}
\begin{proof}
    Lemma~\ref{lemma:joint-continu} together with Proposition~\ref{prop:U-measure} implies that the mapping 
    \begin{equation*}
        \RR^{+} \times \Omega \times AC([-1,0]) \ni (t,\omega,u) \mapsto U_{\omega}(t)\,u \in AC([-1,0])
    \end{equation*}
    is a  Carath\'eodory mapping and is therefore jointly measurable by \cite[Lemma 4.51]{AliB}. 
\end{proof}

\begin{theorem}
    \label{thm:AC-measurable-semiflow}
Let $(\Omega,\mathfrak F,\mathbb P,(\theta_t)_{t\in\mathbb R})$
be a metric dynamical system. Assume that, for every $T>0$,
the coefficient path $(\mu_{\theta_t\omega})_{t\in[0,T]}$ belongs to $\mathfrak m_T$ for every $\omega\in\Omega$, and
that the mapping
\begin{equation*}
    \Omega\ni\omega
    \mapsto
    (\mu_{\theta_t\omega})_{t\in[0,T]}
    \in\mathfrak m_T
\end{equation*}
is $(\mathfrak F,\Sigma_T)$-measurable. Then the solution operators $(U_\omega(t))_{\omega\in\Omega,t\geq0}$
define a measurable linear skew-product semiflow on
$AC([-1,0])$ covering $(\theta_t)_{t\in\mathbb R}$.
\end{theorem}

\begin{proof}
Propositions~\ref{prop:unigue-sol-in-C} and~\ref{prop:Gronwal-determine-bound} ensure that the solution operators
are well defined on $AC([-1,0])$ for every $t\geq0$.
Uniqueness implies that solutions constructed on different
finite time intervals agree on their common domains.

The initial condition gives $U_\omega(0)=\mathrm{Id}_{AC}$,
and Lemma~\ref{lemma:cocycle-U-AC} establishes the cocycle identity. Linearity follows from the linearity of the equation and
uniqueness in Proposition~\ref{prop:unigue-sol-in-C}, while boundedness follows
from Lemma~\ref{lemma:U-cont-Ac}.

Proposition~\ref{prop:U-joint-measure} provides the joint measurability of
$(t,\omega,u)\mapsto U_\omega(t)u$.
Thus all the defining properties of a measurable linear
skew-product semiflow are satisfied.
\end{proof}

\begin{remark}\label{remark-U-notation}
The same solution construction also defines a measurable
linear skew-product semiflow on $C([-1,0])$.
We denote its solution operators by $U_\omega^C(t)$ and
those on $AC([-1,0])$ by $U_\omega^{AC}(t)$.
Existence and uniqueness yield well-defined linear operators,
the identity at time zero, and the cocycle property.
Moreover, the Gr{\"o}nwall estimate gives
\begin{equation*}
    \|U_\omega^C(t)u\|_C
    \leq e^{Kt}\|u\|_C,
    \qquad t\geq0,\quad u\in C([-1,0]).
\end{equation*}

To verify measurability, fix $t\geq0$ and $u\in C([-1,0])$.
Choose a sequence $(u_n)$ in $AC([-1,0])$ converging to $u$
in the supremum norm. If
$i:AC([-1,0])\hookrightarrow C([-1,0])$ denotes the natural
inclusion, uniqueness and the preceding estimate imply
\begin{equation*}
    \|U_\omega^C(t)u-iU_\omega^{AC}(t)u_n\|_C
    \leq e^{Kt}\|u-u_n\|_C
    \to0.
\end{equation*}
Since each mapping
$\omega\mapsto iU_\omega^{AC}(t)u_n$ is measurable,
so is $\omega\mapsto U_\omega^C(t)u$. For each fixed $\omega$, continuity of solution segments
in the supremum norm gives continuity of
$t\mapsto U_\omega^C(t)u$.
Together with the locally uniform operator bound above,
this implies joint continuity in $(t,u)$.
The Carath\'eodory measurability argument used in
Proposition~\ref{prop:U-joint-measure} therefore yields joint measurability of
\begin{equation*}
    [0,\infty)\times\Omega\times C([-1,0])
    \ni(t,\omega,u)\mapsto U_\omega^C(t)u
    \in C([-1,0]).
\end{equation*}
Thus the solution operators define measurable linear
semiflows on both phase spaces, related by
\begin{equation*}
    U_\omega^C(t)i=iU_\omega^{AC}(t),
    \qquad t\geq0.
\end{equation*}
These two semiflows will be considered in the next section,
where we establish their Oseledets decompositions.
\end{remark}

\section{Oseledets decompositions}\label{sec:Oseledets}
In this section, we establish Oseledets decompositions for the
measurable linear semiflows generated by the delay equation.
We first consider the phase space $C([-1,0])$ and apply
\cite[Theorem~3.4]{MierczynskiNovoObaya2020}, together with
the standing assumptions of Section~\ref{sec:Carath-sol}.
In addition to the previously established measurability
properties, this requires suitable integrability estimates
and compactness of the time-one solution operator.
The integrability conditions follow from the Gr{\"o}nwall estimates,
whereas compactness will be proved below using the
Arzel\`a--Ascoli theorem. We also impose the additional
assumption
\begin{equation*}
    \lambda_{\mathrm{top}}^C>-\infty,
\end{equation*}
where $\lambda_{\mathrm{top}}^C$ denotes the top Lyapunov
exponent of the semiflow on $C([-1,0])$.

We then pass to the phase space $AC([-1,0])$ by applying
\cite[Theorem~3.6]{Kryspin2024}.
The key property is that the time-one solution operator maps
continuous initial histories into absolutely continuous
solution segments. Together with appropriate bounds on this
regularization operator and compatibility of the Borel
structures of the two phase spaces, this allows the Oseledets
decomposition to be transferred to $AC([-1,0])$, with the
same Lyapunov exponents. As in the Remark~\ref{remark-U-notation} the corresponding solution operators are denoted by
$U_\omega^C(t)$ and $U_\omega^{AC}(t)$.
\begin{theorem}\label{prop:compact-U-C}
The solution operator
\begin{equation*} 
        U^C_\omega(1):C([-1,0])
    \to C([-1,0]),
    \qquad
    \bigl(U^C_\omega(1)u_0\bigr)(s)
    =z(1+s,\mu,u_0),
\end{equation*}
is compact.
\end{theorem}

\begin{proof}
We show that the image of the
closed unit ball is uniformly bounded and equicontinuous. Let $u_0\in C([-1,0])$ satisfy
$\|u_0\|_{C{([-1,0])}}\le 1$. The integral formulation gives
\begin{equation*}
       z(t,\mu,u_0)
    =
    u_0(0)
    +
    \int_0^t\int_{[-1,0]}
    z(\tau+s,\mu,u_0)\,
    \dd\mu_{\theta_\tau\omega}(s)\,\dd\tau,
    \qquad t\in[0,1].
\end{equation*}
Consequently,
\begin{equation*}
        \sup_{-1\le r\le t}|z(r,\mu,u_0)|
    \le
    1+
    \int_0^t
    \|\mu_{\theta_\tau\omega}\|_{\mathrm{TV}}
    \sup_{-1\le r\le\tau}|z(r,\mu,u_0)|\,\dd\tau.
\end{equation*}
By Gr{\"o}nwall 's inequality,
\begin{equation*}
       \sup_{-1\le r\le t}|z(r,\mu,u_0)|\le\exp\!\Big(\int_0^t\|\mu_{\theta_\tau\omega}\|_{\mathrm{TV}}\,\,\dd\tau
    \Big),
    \qquad t\in[0,1].
\end{equation*}
In particular,
\begin{equation*}
\sup_{\|\varphi\|_{C([-1,0])}\le1}\|U_\omega(1)\varphi\|_{C([-1,0])}\le\exp\!\Big(\int_0^1\|\mu_{\theta_\tau\omega}\|_{\mathrm{TV}}\,\,\dd\tau
    \Big)
    <\infty.
\end{equation*}
Thus, $U^C_\omega(1)$ is bounded, and the image of the unit
ball is uniformly bounded. To prove equicontinuity, take $-1\le a<b\le0$.
Using the integral equation and the preceding estimate,
we obtain
\begin{equation*}
    \begin{split}
    \bigl|
    (U^C_\omega(1)u_0)(b)
        -& (U^C_\omega(1)u_0)(a)
    \bigr|
     =
    \bigl|z(1+b,\mu,u_0)-z(1+a,\mu,u_0)\bigr|\\
    & \le
    \int_{1+a}^{1+b} \|\mu_{\theta_\tau\omega}\|_{\mathrm{TV}}
    \sup_{-1\le r\le\tau}|z(r,\mu,u_0)|\,\dd\tau\\
    & \le
    \exp\!\left( \int_0^1\|\mu_{\theta_\tau\omega}\|_{\mathrm{TV}}\,\dd\tau
    \right)
    \int_{1+a}^{1+b} \|\mu_{\theta_\tau\omega}\|_{\mathrm{TV}}\,\dd\tau.
    \end{split}
\end{equation*}
By absolute continuity of the Lebesgue integral, the
right-hand side tends to zero as $b-a\to0$, uniformly
in $a,b$ and in $u_0$ with $\| u_0\|_{C([-1,0])}\le1$; we are using a classical fact that for any $f\in L_1(\ell)$ and any $\epsilon>0$ there is $\delta>0$ such that for any measurable $E\subset \RR$ such $\ell(E)<\delta$ we have 
\begin{equation*}
    \int_E f\,\dd \ell < \epsilon.
\end{equation*}
Hence, the image of the unit ball is equicontinuous. The Arzel\`a--Ascoli theorem now implies that this image
is relatively compact in $C([-1,0])$. Therefore, $U_\omega(1)$ is a compact operator.
\end{proof}

Assume that the base flow is an ergodic metric dynamical
system on a complete Lebesgue probability space
$(\Omega,\mathfrak F,\mathbb P)$.
We retain the previously imposed measurability assumptions
on the coefficient paths i.e., 
\begin{equation*}
    \Omega\ni\omega
    \mapsto
    (\mu_{\theta_t\omega})_{t\in[0,T]}
    \in\mathfrak m_T
\end{equation*}
is $(\mathfrak F,\Sigma_T)$-measurable.

Before we start we recall measurable projection theorem. It establishes 
measurability of suprema taken over a continuous time interval. See \cite[Theorem~2.12 and Corollary~2.13, pp.~13--14]{Crauel2002} for more details. 

\begin{theorem}[Measurable projection theorem]
\label{thm:measurable-projection}
Let $(\Omega, \mathfrak F,\mathbb P)$ be a complete probability
space, and let $S$ be a Polish space. For every
$A\in \mathfrak F\otimes \mathfrak B(S)$, its projection
\begin{equation*}
    \pi_\Omega(A)
    :=
    \{\omega\in\Omega:
      (\omega,s)\in A \text{ for some }s\in S\}
\end{equation*}
belongs to $ \mathfrak F$.
\end{theorem}
A consequence of the above is that, for every jointly measurable
function
\begin{equation*}
    h:\Omega\times[0,1]\to[0,\infty],
\end{equation*}
the mapping
\begin{equation*}
    \Omega\ni\omega
    \mapsto \sup_{0\leq s\leq1}h(\omega,s)
\end{equation*}
is $ \mathfrak F$-measurable. Indeed, for every $a\in\mathbb R$,
\begin{equation*}
    \big\{\omega:
      \sup_{0\leq s\leq1}h(\omega,s)>a\big\}
    =
    \pi_\Omega
    (\{(\omega,s):h(\omega,s)>a\})
    \in\mathfrak F.
\end{equation*}
We shall apply this observation to jointly measurable functions
given by the logarithms of operator norms.

\begin{theorem}
Suppose that the top Lyapunov exponent of the semiflow
on $C$ satisfies
\begin{equation*}
    \lambda_{\mathrm{top}}^C>-\infty.
\end{equation*}
Then this semiflow admits an Oseledets decomposition
in the sense of
\cite[Definition~3.1]{MierczynskiNovoObaya2020}.
\end{theorem}

\begin{proof}
We apply \cite[Theorem~3.4]{MierczynskiNovoObaya2020},
including the standing assumptions of Section~\ref{sec:Carath-sol}
of that paper. The space $C$ is a separable Banach space, and the solution
operators form a measurable linear semiflow by the results
of the preceding section.
The Gr{\"o}nwall  estimate gives $  \|U_\omega^C(t)\|_{\mathcal L(C)}
    \leq e^{KT}$. Consequently,
\begin{equation*}
    \sup_{0\leq s\leq1}
    \log^+\|U_\omega^C(s)\|_{\mathcal L(C)} 
    \leq K, \quad\text{and}\quad \sup_{0\leq s\leq1}
    \log^+\|U_{\theta_s\omega}^C(1-s)\|_{\mathcal L(C)}
    \leq K.
\end{equation*}
These suprema are measurable on the complete Lebesgue
base by the measurable projection theorem~\ref{thm:measurable-projection}, applied to
the jointly measurable operator norms.
Thus both functions belong to $L^1(\Omega,\mathfrak F,\mathbb P)$.

The operator $U_\omega^C(1)$ is compact by the compactness
result proved in Theorem~\ref{prop:compact-U-C}. Together with the assumption
$\lambda_{\mathrm{top}}^C>-\infty$, these observations
verify all the hypotheses of the cited theorem.
\end{proof}

The passage to $AC$ requires more than its continuous
inclusion in $C$: one must also control the regularization
operator and verify compatibility of the measurable
structures. These are the properties used
in \cite[Theorem~3.6]{Kryspin2024}.

Before proceeding to the transfer of the Oseledets decomposition, we recall Lusin--Souslin theorem. Lusin--Souslin theorem ensures
compatibility of the Borel structures under the inclusion
$AC([-1,0])\hookrightarrow C([-1,0])$ and allows measurability
to be recovered in the stronger phase space. See \cite[Theorem~15.1]{Kechris1995} for more details. 

\begin{theorem}[Lusin--Souslin]
\label{thm:lusin-souslin}
Let $X$ and $Y$ be Polish spaces, and let $f:X\to Y$ be
continuous. If $A\subseteq X$ is Borel and the restriction
$f{\restriction}_A$ is injective, then $f(A)$ is a Borel subset of $Y$.
\end{theorem}
In particular, if $i:X\to Y$ is a continuous injection between
Polish spaces, then
\begin{equation*}
    \mathfrak B(X)
    =
    \{i^{-1}(B):B\in \mathfrak B(Y)\}.
\end{equation*}
Indeed, continuity gives one inclusion, while the other follows
from
\begin{equation*}
    A=i^{-1}(i(A)),
    \qquad A\in \mathfrak B(X),
\end{equation*}
since $i(A)\in \mathfrak B(Y)$ by
Theorem~\ref{thm:lusin-souslin}.
Consequently, for any measurable space $(S,\Sigma)$ and any
mapping $g:S\to X$, the mapping $g$ is
$(\Sigma, \mathfrak B(X))$-measurable if and only if $i\circ g$
is $(\Sigma, \mathfrak B(Y))$-measurable.

\begin{theorem}
Under the assumption $\lambda_{\mathrm{top}}^C>-\infty$,
the semiflow on $AC$ admits an Oseledets decomposition
with the same Lyapunov exponents
as the semiflow on $C$. More precisely, let
\begin{equation*}
    i:AC\to C,
    \qquad i(u)=u,
\end{equation*}
be the natural inclusion.
On a common invariant set of full measure, the
finite-dimensional Oseledets subspaces satisfy $iE_j^{AC}(\omega)=E_j^C(\omega)$,
and the corresponding filtration spaces satisfy $F_j^{AC}(\omega)=i^{-1}(F_j^C(\omega))$. 
In the case of finitely many finite Lyapunov exponents,
the residual space satisfies $ F_\infty^{AC}(\omega)=i^{-1}(F_\infty^C(\omega))$. 
\end{theorem}

\begin{proof}
We verify assumptions (A1)--(A6) of
\cite[Section~3]{Kryspin2024}, with
$X_1=C$ and $X_2=AC$. The natural inclusion $i$ is linear, injective, and bounded,
since
\begin{equation*}
    \|iu\|_C
    \leq
    |u(-1)|+\int_{-1}^{0}|u'(s)|\,\dd s
    =
    \|u\|_{AC}.
\end{equation*}
Uniqueness of solutions gives
\begin{equation*}
    U_\omega^C(t)i=iU_\omega^{AC}(t),
    \qquad t\geq0.
\end{equation*}
Thus (A1) and (A2) hold. To verify (A3), define the regularization operator
\begin{equation*}
    R_\omega:C\to AC,
    \qquad
    (R_\omega u)(s):=z(1+s;\mu,u),
    \quad s\in[-1,0].
\end{equation*}
This operator is well defined because the solution
is absolutely continuous on $[0,1]$, even when the
initial history is only continuous.
Moreover,
\begin{equation*}
    iR_\omega=U_\omega^C(1)
    \quad \text{and} \quad
    R_\omega i=U_\omega^{AC}(1).
\end{equation*}
The solution estimate and the equation imply
\begin{equation*}
    \begin{aligned}
        \|R_\omega u\|_{AC}
        &=
        |u(0)|+\int_0^1|z'(t;\mu,u)|\,\dd t\\
        &\leq
        \|u\|_C+
        \int_0^1 K e^{Kt}\|u\|_C\,\dd t\\
        &=e^K\|u\|_C.
    \end{aligned}
\end{equation*}
Hence $R_\omega\in\mathcal L(C,AC)$ and
\begin{equation*}
    \sup_{0\leq s\leq1}
    \log^+\|R_{\theta_s\omega}\|_{\mathcal L(C,AC)}
    \leq K.
\end{equation*}
The measurability of this supremum follows from the
joint measurability of the operator and the
measurable projection theorem on the complete Lebesgue
base. The required integrability in (A3) follows. Both $C$ and $AC$ are separable Banach spaces. Thus (A5) holds as well. We next check (A4) and (A6).

Since $i$ is a continuous injection between Polish spaces,
the Lusin--Souslin theorem implies that $i(A)$ is Borel
in $C$ for every Borel set $A\subseteq AC$.
In particular, $A=i^{-1}(i(A))$, 
which proves (A6).

For each fixed $u\in C$, measurability of
$\omega\mapsto U_\omega^C(1)u$ and the identity
$iR_\omega u=U_\omega^C(1)u$ give
\begin{equation*}
    \{\omega:R_\omega u\in A\}
    =
    \{\omega:U_\omega^C(1)u\in i(A)\}
    \in\mathfrak F
\end{equation*}
for every Borel set $A\subseteq AC$.
This proves the strong measurability required in (A4).
The same argument with an additional time parameter
justifies the joint measurability used above.

All assumptions of
\cite[Theorem~3.6]{Kryspin2024}
are therefore satisfied, and the Oseledets decomposition
on $C$ transfers to $AC$. For completeness, the transferred subspaces are given by
\begin{equation*}
    E_j^{AC}(\omega)
    :=
    R_{\theta_{-1}\omega}E_j^C(\theta_{-1}\omega).
\end{equation*}
Indeed, invariance of the Oseledets subspaces yields
\begin{equation*}
    \begin{aligned}
        iE_j^{AC}(\omega)
        &=
        iR_{\theta_{-1}\omega}
        E_j^C(\theta_{-1}\omega)\\
        &=
        U_{\theta_{-1}\omega}^C(1)
        E_j^C(\theta_{-1}\omega)\\
        &=
        E_j^C(\omega).
    \end{aligned}
\end{equation*}
This is the construction in
\cite[Claim~3.4 and Theorem~3.6]{Kryspin2024}.
Since $i$ is injective, it also gives
\begin{equation*}
    \dim E_j^{AC}(\omega)=\dim E_j^C(\omega).
\end{equation*}
The equality of growth rates can also be seen directly.
For $u\in AC$ and $t\geq1$, the cocycle identities imply
\begin{equation*}
    U_\omega^{AC}(t)u
    =
    R_{\theta_{t-1}\omega}
    U_\omega^C(t-1)iu.
\end{equation*}
Consequently,
\begin{equation*}
    \|U_\omega^C(t)iu\|_C
    \leq
    \|U_\omega^{AC}(t)u\|_{AC}
    \leq
    e^K\|U_\omega^C(t-1)iu\|_C.
\end{equation*}
For nonzero $u\in E_j^{AC}(\omega)$, the first and last
expressions have logarithmic growth rate $\lambda_j$.
Therefore,
\begin{equation*}
    \lim_{t\to\infty}
    \frac{1}{t}\log\|U_\omega^{AC}(t)u\|_{AC}
    =
    \lambda_j.
\end{equation*}
The same estimate preserves growth rate $-\infty$.
The filtration identities, measurability, and temperedness
of the associated projections follow from the cited
transfer theorem.
\end{proof}

The assumption $\lambda_{\mathrm{top}}^C>-\infty$
must be verified separately or retained as an explicit
hypothesis. The bound on the total variation provides
the required upper growth estimates, but does not
exclude superexponential decay.

\begin{example}
Finite-time extinction may occur even for a deterministic
equation with uniformly bounded measure-valued coefficients.
Consider
\begin{equation*}
    z'(t)=-z(\lfloor t\rfloor),\qquad t\geq0,
\end{equation*}
with initial history $u_0\in C([-1,0])$.
This equation has the required measure representation with
\begin{equation*}
    \mu_t=-\delta_{\lfloor t\rfloor-t},
    \qquad
    \|\mu_t\|_{\mathrm{TV}}=1.
\end{equation*}
Since $\lfloor t\rfloor-t\in(-1,0]$, the delay remains
within the prescribed interval. For $t\in[n,n+1]$, integration of the equation gives
\begin{equation*}
    z(t)=(1-(t-n))z(n).
\end{equation*}
The
equation on $[0,1)$ reduces to $z'(t)=-u_0(0)$. Hence for $f\in [0,1)$ we have 
\begin{equation*}
    z(t)=(1-t)u_0(0) 
\end{equation*}
Thus $z(1)=0$, regardless of the value of $u_0(0)$.
On each subsequent interval $[n,n+1)$, the derivative
equals $-z(n)$. Induction therefore shows that
$z(t)=0$ for every $t\geq1$. Let $U(t,0)$ denote the solution operator from initial
time $0$ to time $t$, defined by
\begin{equation*}
    \bigl(U(t,0)u_0\bigr)(s)=z(t+s),
\end{equation*}
for $s\in[-1,0]$. 
Although the solution vanishes from time $1$ onward,
its entire history segment is guaranteed to vanish
only from time $2$ onward, when $[t-1,t]\subseteq[1,\infty)$.
Since this holds for every initial history, we obtain
\begin{equation*}
    U(t,0)=0\quad\text{for all }t\geq2,
    \qquad
    \lim_{t\to\infty}\frac{1}{t}\log\|U(t,0)\|=-\infty,
\end{equation*}
where $\log0=-\infty$. Thus, boundedness of the coefficients and well-posedness
do not by themselves exclude a growth rate of $-\infty$.
The assumption $\lambda_{\mathrm{top}}>-\infty$ excludes
this type of degeneracy.
\end{example}

\section*{Data availability}

No datasets were generated or analysed in this study.

\end{document}